\documentclass{amsart}

\usepackage{amsmath, amssymb}
\theoremstyle{plain}

\newtheorem{thm}{Theorem}[section]
\newtheorem{prop}[thm]{Proposition}
\newtheorem{lem}[thm]{Lemma}
\newtheorem{cor}[thm]{Corollary}

\theoremstyle{definition}
\newtheorem{defn}[thm]{Definition}
\newtheorem{exam}[thm]{Example}
\newtheorem{rem}[thm]{Remark}
\newtheorem{conjecture}[thm]{Conjecture}

\numberwithin{equation}{section}

\def\Hom{\operatorname{Hom}}   
\def\deg{\operatorname{deg}}   \def\Supp{\operatorname{Supp}}
\def\HF{\operatorname{HF}}     \def\HP{\operatorname{HP}}
\def\ri{\operatorname{ri}}

\newcommand{\bbQ}{\ensuremath{\mathbb Q}}
\newcommand{\bbZ}{\ensuremath{\mathbb Z}}
\newcommand{\bbN}{\ensuremath{\mathbb N}}
\newcommand{\bbP}{\ensuremath{\mathbb P}}
\newcommand{\bbX}{\ensuremath{\mathbb X}}
\newcommand{\bbY}{\ensuremath{\mathbb Y}}

\newcommand{\fm}{\ensuremath{\mathfrak{m}}}
\newcommand{\fC}{\mathfrak{C}}
\newcommand{\fG}{\mathfrak{G}}
\newcommand{\fP}{\mathfrak{P}}
\newcommand{\calO}{{\mathcal{O}}}

\begin{document}

\title[On the Dedekind Different of a Cayley-Bacharach Scheme]
{On the Dedekind Different of a Cayley-Bacharach Scheme}

\author{Martin Kreuzer}
\address[Martin Kreuzer]{Fakult\"{a}t f\"{u}r Informatik und Mathematik \\
Universit\"{a}t Passau, D-94030 Passau, Germany}
\email{martin.kreuzer@uni-passau.de}

\author{Tran N.K. Linh}
\address[Tran N.K. Linh]{
Research Institute for Symbolic Computation (RISC) \\
J. Kepler University, Altenbergerstr. 69, A-4040 Linz, Austria
\newline
\hspace*{.5cm} \textrm{and} Department of Mathematics,
Hue University's College of Education,
34 Le Loi, Hue, Vietnam}
\email{tnkhanhlinh141@gmail.com}

\author{Le Ngoc Long}
\address[Le Ngoc Long]{Fakult\"{a}t f\"{u}r Informatik und Mathematik \\
Universit\"{a}t Passau, D-94030 Passau, Germany \newline
\hspace*{.5cm} \textrm{and} Department of Mathematics,
Hue University's College of Education,
34 Le Loi, Hue, Vietnam}
\email{nglong16633@gmail.com}



\date{\today}

\dedicatory{}

\commby{Martin Kreuzer, Tran N.K. Linh, and Le Ngoc Long} %

% -----------------------------------------------------------

\begin{abstract}
Given a 0-dimensional scheme $\bbX$ in a projective space
$\bbP^n_K$ over a field $K$, we characterize the
Cayley-Bacharach property of $\bbX$ in terms of the
algebraic structure of the Dedekind different of
its homogeneous coordinate ring. Moreover, we characterize
Cayley-Bacharach schemes by Dedekind's formula
for the conductor and the complementary module, 
we study schemes with minimal Dedekind different using
the trace of the complementary module, and we prove
various results about almost Gorenstein and 
nearly Gorenstein schemes.
\end{abstract}

% -----------------------------------------------------------
\maketitle
% -----------------------------------------------------------

\section{Introduction}

Let $K$ be a field, and let $\bbP^n_K$ be the $n$-dimensional
projective space over $K$. We are interested in studying
0-dimensional subschemes $\bbX$ of $\bbP^n_K$.
Classically, the Cayley-Bacharach property of a reduced
scheme has been defined to mean that all hypersurfaces
of a certain degree which pass through all points of~$\bbX$
but one automatically pass through the last point.
Here we generalize this definition to arbitrary 0-dimensional
subschemes of~$\bbP^n_K$ over an arbitrary field $K$.
In~\cite{GKR}, Geramita \emph{et al.} used
the canonical module $\omega_R$ of the homogeneous coordinate
ring $R$ of~$\bbX$ to characterize the Cayley-Bacharach
property algebraically when $\bbX$ is reduced and
$K$ is algebraically closed.
Later, in \cite{Kr1} and \cite{KL}, this result was generalized
to arbitrary 0-dimensional schemes with $K$-rational support.

In this paper we use the Dedekind different to study the
Cayley-Bacharach property. The Dedekind different
$\delta_\bbX^\sigma$ of~$R$ is the inverse ideal of its
Dedekind complementary module $\fC_\bbX^\sigma$ in its
homogeneous ring of quotients $Q^h(R)$. Here the module
$\fC_\bbX^\sigma$ is a fractional ideal of $Q^h(R)$,
which is defined if $\bbX$ is locally Gorenstein, and
$\sigma$ is a fix homogeneous trace map.
Theorem~\ref{thmSec3.5}, one of our main results,
characterizes Cayley-Bacharach schemes, i.e.,
schemes having the Cayley-Bacharach property of
maximal degree $r_\bbX-1$, in terms of the structure of
their Dedekind different $\delta_\bbX^\sigma$.
Another main result, Theorem~\ref{thmSec4.7},
characterizes Cayley-Bacharach schemes as the ones
for which Dedekind's formula for the conductor and
the Dedekind complementary module holds true.
Applications include several characterizations
of schemes $\bbX$ with minimal Dedekind different and
a characterization of almost Gorenstein scheme $\bbX$
by the nearly Gorenstein and the Cayley-Bacharach
properties.

In the following we describe the contents of the paper
in more detail. Section~2 starts by recalling the notion of
maximal subschemes, minimal separators and the maximal
degree of a minimal separator. We describe the Hilbert function
of a maximal subscheme of~$\bbX$, define standard sets of
separators, and use them to control the ring structure of~$R$
in degrees $\ge r_\bbX$, where $r_\bbX$ is the regularity
index of~$\bbX$.

Next, in Section~3, we rework the construction of the
Dedekind complementary module $\fC_\bbX^\sigma$ from the
local case given in~\cite{HKW}. Then we work out explicit
descriptions of its homogeneous components and its Hilbert
function.
As mentioned above, the Dedekind different $\delta_\bbX^\sigma$
is defined as the inverse ideal of~$\fC_\bbX^\sigma$.
We provide its Hilbert function, Hilbert polynomial, and
a sharp bound for its regularity index.
If the containment
$\bigoplus_{i\ge 2r_\bbX} R_i\subseteq \delta_\bbX^\sigma$
is an equality, we say that $\bbX$ has minimal Dedekind different.
For reduced schemes $\bbX$ in~$\bbP^2_K$, we show that
this condition implies that $\delta_\bbX^\sigma$ agrees
with the K\"{a}hler different of~$\bbX$.

Section 4 starts with the general definition of the
Cayley-Bacharach property of degree $d$ (in short, CBP($d$))
and of Cayley-Bacharach schemes.
The main result of this section is Theorem~\ref{thmSec3.5}.
It shows that a $0$-dimensional locally Gorenstein scheme
$\bbX$ is a Cayley-Bacharach scheme if any only if
the Dedekind different $\delta^\sigma_\bbX$ satisfies
\begin{equation}\tag{$\ast$}
x_0^{r_{\bbX}-1}(I_{\bbY/\bbX})_{r_{\bbX}}
\nsubseteq (\delta^\sigma_\bbX)_{2r_{\bbX}-1}
\end{equation}
for all $p_j\in\Supp(\bbX)$ and every maximal $p_j$-subscheme
$\bbY \subseteq \bbX$.
This theorem allows us to detect Cayley-Bacharach schemes
by looking at a single homogeneous component of the Dedekind
different. Moreover, we can describe the growth of the
Hilbert function of the Dedekind different of
a Cayley-Bacharach scheme and determine its regularity
index (see Proposition~\ref{propSec3.8}).
A property similar to $(\ast)$ allows us to detect
the Cayley-Bacharach property of any degree
(see Proposition~\ref{propSec3.10}), but is not
equivalent to it in general (see Example~\ref{examSec3.11}).

In Section 5 we look at the conductor
$\mathfrak{F}_{\widetilde{R}/R}$ of $R$ in the ring
$\widetilde{R}=\prod_{i=1}^s \calO_{\bbX,p_i}[T_i]$,
where $T_1,\dots,T_s$ are indeterminates.
If $\bbX$ is reduced, this is the classical conductor
of $R$ in its integral closure. After showing a chain
of inclusions
$\mathfrak{F}_{\widetilde{R}/R}^2\subseteq
\delta_\bbX^\sigma \subseteq
\mathfrak{F}_{\widetilde{R}/R}$
between the conductor and the Dedekind different,
we generalize a result of~\cite{GKR} which characterizes
Cayley-Bacharach schemes in terms of their conductors.
More precisely, we prove that the Cayley-Bacharach
property of degree $d$ is equivalent to
$\mathfrak{F}_{\widetilde{R}/R} \subseteq
\bigoplus_{i\ge d+1} R_i$, and that $\bbX$ is a
Cayley-Bacharach scheme if and only if
$\mathfrak{F}_{\widetilde{R}/R} =
\bigoplus_{i\ge r_\bbX} R_i$
(see Theorem~\ref{thmSec4.4}).
A further main result is the generalization
of Dedekind's formula
$\mathfrak{F}_{\widetilde{R}/R}\cdot\fC_\bbX^\sigma
=\widetilde{R}$
for the conductor and the Dedekind complementary module
given in Theorem~5.7. These theorems have a number
of applications to schemes with minimal Dedekind different,
to locally Gorenstein schemes, and to
Cayley-Bacharach schemes (see Proposition~\ref{propSec4.5},
Corollary~\ref{corSec4.09} and Corollary~\ref{corSec4.10}).

In the last section we use the trace of the
Dedekind complementaty module to characterize schemes
with minimal Dedekind different by the Cayley-Bacharach
property and by
$\mathfrak{F}_{\widetilde{R}/R}={\rm tr}(\fC_\bbX^\sigma)$.
Moreover, we provide a number of contributions to the topics
of nearly Gorenstein and almost Gorenstein schemes which
have received some attention lately (see \cite{BF97, GTT, HHS}).
Among others, we prove an analogue of~\cite[Proposition~6.1]{HHS}
in our setting, which characterizes almost Gorenstein schemes
by the nearly Gorenstein property and one value of
the Hilbert function of the Dedekind different.
Further applications to the case
$\Delta_\bbX=\HF_\bbX(r_\bbX)-\HF_\bbX(r_\bbX-1)=1$,
to Cayley-Bacharach schemes, and to level schemes follow.
In particular, we point out that
every almost Gorenstein scheme is nearly Gorenstein.
In the case $\Delta_\bbX=1$,
the converse of this property holds true if $\bbX$ is
a Cayley-Bacharach scheme (see Proposition~\ref{propSec5.5}).
Moreover, we show that $\bbX$ is almost Gorenstein
if and only if it is a Cayley-Bacharach scheme and
$\HF_{\delta^\sigma_\bbX}(r_\bbX+1)= \HF_\bbX(1)$
(see Proposition~\ref{propSec5.8}), and provide
a different proof of a result in~\cite[10.2-4]{GTT}
when the graded ring has dimension one.
In our setting, this result states that a 0-dimensional locally
Gorenstein scheme with $\Delta_\bbX\ge 2$ is level
and almost Gorenstein if and only if $r_\bbX=1$.
Finally, we show that an almost Gorenstein
$(2,r_\bbX-1)$-uniform set $\bbX$ of distinct
$K$-rational points with $r_\bbX\ge 2$ satisfies $\Delta_\bbX=1$.

Unless explicitly mentioned otherwise, we use the definitions
and notation introduced in the books \cite{KR1, KR2, Ku2}.
All examples in this paper were calculated by using
the computer algebra system ApCoCoA (see~\cite{ApC}).

\medskip\bigbreak
\section{Separators of Maximal $p_j$-Subschemes}

Throughout the paper, we work over an arbitrary field $K$.
By $\bbP^n_K$ we denote the projective $n$-space over $K$.
The homogeneous coordinate ring of $\bbP^n_K$ is the polynomial
ring $P=K[X_0,\dots,X_n]$ equipped with the standard grading.
We are interested in studying a 0-dimensional subscheme $\bbX$
of~$\bbP^n_K$. Its homogeneous vanishing ideal in $P$ is denoted
by~$I_\bbX$. The homogeneous coordinate ring of~$\bbX$ is then
given by $R:=P/I_\bbX.$  The ring $R$ is a standard
graded $K$-algebra. Its homogeneous maximal ideal is
denoted by $\fm$.

The set of closed points of~$\bbX$ is called the \textbf{support}
of~$\bbX$ and is denoted by $\Supp(\bbX) = \{p_1,\dots,p_s\}$.
Once and for all, we assume that no point of the support of~$\bbX$
lies on the hyperplane at infinity $\mathcal{Z}(X_0)$.
Consequently, the residue class $x_0$ of $X_0$ in $R$
is a non-zerodivisor and $R$ is a 1-dimensional
Cohen-Macaulay ring.
To each point $p_j \in \Supp(\bbX)$ we have the associated
local ring $\calO_{\bbX,p_j}$. Its maximal ideal is denoted
by $\fm_{\bbX,p_j}$, and the residue field of~$\bbX$ at~$p_j$
is denoted by~$\kappa(p_j)$.
The {\bf degree} of~$\bbX$ is defined as
$\deg(\bbX) = \sum_{j=1}^s \dim_K(\calO_{\bbX,p_j})$.
Furthermore, the {\bf homogeneous ring of quotients} of~$R$,
denoted by $Q^h(R)$, is defined as the localization
of~$R$ with respect to the set of all homogeneous
non-zerodivisors of~$R$.
In view of~\cite[Proposition~3.1]{KL},
there are isomorphisms of graded $R$-modules
$$
Q^h(R) \cong
{\textstyle \prod\limits_{j=1}^s} \calO_{\bbX,p_j}[T_j,T_j^{-1}]
\cong R_{x_0}
$$
where $T_1,\dots,T_s$
are indeterminates with $\deg(T_1) = \cdots = \deg(T_s)=1$.

The following special class of subschemes of the scheme $\bbX$
plays an important role in this paper.

\begin{defn}
Let $j\in\{1,\dots,s\}$. A subscheme $\bbY\subseteq\bbX$ is called
a~\textbf{$p_j$-subscheme}
if the following conditions are satisfied:
\begin{enumerate}
  \item $\calO_{\bbY,p_k} = \calO_{\bbX,p_k}$ for $k \ne j$.
  \item The map
  $\calO_{\bbX,p_j}\twoheadrightarrow\calO_{\bbY,p_j}$
  is an epimorphism.
\end{enumerate}

A $p_j$-subscheme $\bbY\subseteq\bbX$ is called \textbf{maximal}
if $\deg(\bbY) = \deg(\bbX)-\dim_K \kappa(p_j)$.
\end{defn}

If $\bbX$ has $K$-rational support (i.e., all closed points
of $\bbX$ are $K$-rational), then a subscheme $\bbY \subseteq \bbX$
of degree $\deg(\bbY) = \deg(\bbX)-1$ with $\calO_{\bbY,p_j} \ne
\calO_{\bbX,p_j}$ is exactly a maximal $p_j$-subscheme of~$\bbX$.

A relationship between maximal $p_j$-subschemes
of~$\bbX$ and ideals of the product of local rings
can be described as follows (cf.~\cite[Proposition~3.2]{KL}).

\begin{prop}\label{propSec1.2}
Let $\Gamma=\prod_{j=1}^s\calO_{\bbX,p_j}$,
and let $\fG(\calO_{\bbX,p_j}) =
{\rm Ann}_{\calO_{\bbX,p_j}}(\fm_{\bbX,p_j})$
be the socle of $\calO_{\bbX,p_j}$.
There is a 1-1 correspondence
$$
\bigg\{
\begin{array}{c}
\textrm{maximal}\ p_j\textrm{-subschemes}\\
\textrm{of the scheme}\ \bbX
\end{array}
\bigg\}
\longleftrightarrow
\bigg\{
\begin{array}{c}
\textrm{ideals}\ \langle (0,\dots,0,s_j,0,\dots,0) \rangle_{\Gamma}
\subseteq \Gamma\\
\textrm{with}\ s_j\in \fG(\calO_{\bbX,p_j})\setminus\{0\}
\end{array}
\bigg\}.
$$
\end{prop}

Let $\bbY$ be a maximal $p_j$-subscheme of $\bbX$, let
$I_{\bbY/\bbX}$ be the ideal of $\bbY$ in $R$, and let
$\alpha_{\bbY/\bbX} :=
\min\{i\in \bbN \mid (I_{\bbY/\bbX})_i \ne \langle 0\rangle\}$.
Furthermore, we let $s_j\in \fG(\calO_{\bbX,p_j})\setminus\{0\}$
be a socle element corresponding to $\bbY$.
Then there is a non-zero homogeneous element
$f_\bbY \in (I_{\bbY/\bbX})_i$, $i \ge \alpha_{\bbY/\bbX}$,
such that
$\tilde{\imath}(f_\bbY) = (0,\dots,0, s_jT_j^i,0,\dots,0)$.
Here the injection
$$
\tilde{\imath}: R \longrightarrow Q^h(R) \cong
{\textstyle\prod\limits_{j=1}^s}\calO_{\bbX,p_j}[T_j,T_j^{-1}]
$$
is the homogeneous map of degree zero given by
$\tilde{\imath}(f) = (f_{p_1}T_1^i,\dots, f_{p_s}T_s^i)$,
for $f \in R_i$ with $i\ge 0$, where $f_{p_j}$ is the germ
of $f$ at the point $p_j$ of $\Supp(\bbX)$.

Let $\varkappa_j := \dim_K \kappa(p_j)$, and let
$\{e_{j1}, \dots, e_{j\varkappa_j}\}\subseteq \calO_{\bbX,p_j}$
be elements whose residue classes form a $K$-basis of $\kappa(p_j)$.
For $a \in \calO_{\bbX,p_j}$ and for $k_j = 1, \dots, \varkappa_j$,
we set
$$
\mu(a) := \min\{i\in\bbN \,\mid\,
(0,\dots,0,aT_j^i,0,\dots,0) \in \tilde{\imath}(R)\}
$$
and
$$
f^*_{jk_j} := \tilde{\imath}^{-1}((0,\dots,0,
e_{jk_j}s_{j}T_j^{\mu(e_{jk_j}s_{j})},0,\dots,0)).
$$

\begin{defn}
Let $\bbY$ be a maximal $p_j$-subscheme as above.
\begin{enumerate}
 \item The set $\{f^*_{j1},\dots,f^*_{j\varkappa_j}\}$ is called
  the \textbf{set of minimal separators of $\bbY$ in $\bbX$}
  with respect to $s_j$ and $\{e_{j1},\dots,e_{j\varkappa_j}\}$.

 \item The number
   $$
   \mu_{\bbY/\bbX}:=\max\{\, \deg(f^*_{jk_j})
   \mid k_j = 1, \dots, \varkappa_j \,\}
   $$
   is called the
   \textbf{maximal degree of a minimal separator of~$\bbY$ in~$\bbX$}.
\end{enumerate}
\end{defn}

\begin{rem} \label{remSec1.4}
Let $\bbY$ be a maximal $p_j$-subscheme of~$\bbX$.
\begin{enumerate}
  \item[(a)]
  The maximal degree of a minimal separator of~$\bbY$ in~$\bbX$
  depends neither on the choice of the socle element~$s_j$
  nor on the specific choice of~$\{e_{j1},\dots,e_{j\varkappa_j}\}$
  (see \cite[Lemma~3.4]{KL}).

  \item[(b)]
  Set $U := \langle(0,\dots,0,s_j,0,\dots,0)\rangle_{Q^h(R)}$.
  As in the proof of~\cite[Proposition~4.2]{KL},
  we have $I_{\bbY/\bbX} = \tilde{\imath}^{-1}(U)$
  and $\dim_K(I_{\bbY/\bbX})_i = \dim_K U_i = \varkappa_j$
  for $i\gg0$. In particular,
  $f^*_{j1},\dots,f^*_{j\varkappa_j} \in I_{\bbY/\bbX}$.

  \item[(c)]
  If $\bbX$ has $K$-rational support, then
  $\varkappa_1=\cdots=\varkappa_s=1$ and a minimal
  separator $f^*_{\bbY}$ of~$\bbY$ in~$\bbX$ is nothing but
  a non-zero element of~$(I_{\bbY/\bbX})_{\alpha_{\bbY/\bbX}}$,
  i.e., $f^*_{\bbY}$ is a minimal separator of~$\bbY$ in~$\bbX$
  in the sense of~\cite{Kr1}.
\end{enumerate}
\end{rem}

Now we examine the Hilbert function of a maximal $p_j$-subscheme
of~$\bbX$. Recall that the {\bf Hilbert function} of a finitely
generated graded $R$-module $M$ is a map $\HF_M: \bbZ \rightarrow \bbN$
given by $\HF_M(i)=\dim_K(M_i)$.
The unique polynomial $\HP_M(z) \in \mathbb{Q}[z]$ for which
$\HF_M(i)=\HP_M(i)$ for all $i\gg 0$ is called the
\textbf{Hilbert polynomial} of $M$.
The number
$$
\ri(M)=\min\big\{\, i\in\bbZ \mid \HF_{M}(j)=\HP_M(j)\
\textrm{for all} \ j\geq i \,\big\}
$$
is called the \textbf{regularity index} of $M$ (or of $\HF_M$).
Whenever $\HF_M(i)=\HP_M(i)$ for all $i\in\bbZ$, we let $\ri(M)=-\infty.$
Instead of $\HF_R$ we also write $\HF_{\bbX}$ and call it the
Hilbert function of~$\bbX$. Its regularity index is denoted by~$r_{\bbX}.$
Note that $\HF_{\bbX}(i)=0$ for $i<0$ and
$$
1=\HF_{\bbX}(0)<\HF_{\bbX}(1)<\cdots<\HF_{\bbX}(r_{\bbX}-1)< \deg(\bbX)
$$
and $\HF_{\bbX}(i)=\deg(\bbX)$ for $i\geq r_{\bbX}$.

\begin{prop}\label{propSec1.5}
Let $\bbY\subseteq\bbX$ be a maximal $p_j$-subscheme, let $s_j$ be
a socle element of $\calO_{\bbX,p_j}$ corresponding to $\bbY$, let
$\{e_{j1},\dots,e_{j\varkappa_j}\}\subseteq\calO_{\bbX,p_j}$
be elements whose residue classes form a $K$-basis of~$\kappa(p_j)$,
and let $\{f^*_{j1},\dots,f^*_{j\varkappa_j}\}$
be the set of minimal separators of~$\bbY$ in~$\bbX$ with respect
to~$s_j$ and $\{e_{j1},\dots,e_{j\varkappa_j}\}$.
Then the following assertions hold true.

\begin{enumerate}
 \item We have
  $I_{\bbY/\bbX}=\langle f \rangle_R^{\mathrm{sat}}$
  for every $f\in (I_{\bbY/\bbX})_{i}\setminus\{0\}$
  with $i \ge \alpha_{\bbY/\bbX}$,
  where $\langle f \rangle_R^{\mathrm{sat}}
  = \{\, g\in R \mid \mathfrak{m}^ig \subseteq \langle f \rangle_R
  \mbox{ for some $i\ge 0$} \,\}$
  is the saturation of $\langle f \rangle_R$.

 \item We have $\alpha_{\bbY/\bbX}\le\mu_{\bbY/\bbX} \le r_{\bbX}$
 and the Hilbert function of~$\bbY$ satisfies
  $$
  \HF_{\bbY}(i)=
  \begin{cases}
  \HF_{\bbX}(i) & \textrm{if}\quad i < \alpha_{\bbY/\bbX},\\
    \le\HF_{\bbX}(i)-1 &  \textrm{if}\quad
    \alpha_{\bbY/\bbX} \le i < \mu_{\bbY/\bbX},\\
    \HF_{\bbX}(i)-\varkappa_j  &
    \textrm{if}\quad i \ge \mu_{\bbY/\bbX}.
  \end{cases}
  $$

 \item There is a special choice of a set
  $\{\, e_{j1}, \dots, e_{j\varkappa_j} \,\}\subseteq\calO_{\bbX,p_j}$
  such that its residue classes form a $K$-basis of $\kappa(p_j)$,
  $I_{\bbY/\bbX}\!=\!\langle f^*_{j1},\dots,f^*_{j\varkappa_j}\rangle_R$,
  and for all $i\in\bbZ$ we have
  $$
  \Delta\HF_{\bbY}(i) =
  \Delta\HF_{\bbX}(i) - \#\big\{\, k \in \{1,\dots,\varkappa_j\}
  \mid \deg(f^*_{jk})=i \,\big\}.
  $$
\end{enumerate}
\end{prop}

\begin{proof}
(a)\ It is clear that
$\langle f \rangle_R\subseteq \langle f \rangle_R^{\mathrm{sat}}
\subseteq I_{\bbY/\bbX}$.
For the other inclusion, we use Remark~\ref{remSec1.4}(b)
and write
$$
\tilde{\imath}(f)=(0,\dots,0,as_{j}T_j^{i},0,\dots,0)
\in Q^h(R)
$$
for some $a \in \calO_{\bbY,p_j}\setminus\fm_{\bbX,p_j}$.
Similarly, for every $g\in (I_{\bbY/\bbX})_{k}$ with
$k\geq \alpha_{\bbY/\bbX}$
we have $\tilde{\imath}(g)=(0,\dots,0,bs_{j}T_j^k,0,\dots,0)$
with $b \in \calO_{\bbX,p_j}$.
If $b$ is not a unit of $\calO_{\bbX,p_j}$, then $bs_{j} = 0$,
and so $g = 0\in \langle f \rangle_R^{\mathrm{sat}}$.
Otherwise, since $R_{i+r_\bbX} \cong Q^h(R)_{i+r_\bbX}$
for all $i \ge 0$, we let
$$
h = \tilde{\imath}^{-1}
((0,\dots,0,ba^{-1}T_j^{r_{\bbX}},0,\dots,0))\in R_{r_{\bbX}}.
$$
Then we have
$x_0^{r_{\bbX}+i}\!g = x_0^khf\in\! \langle f \rangle_R$, and
consequently $g \!\in\! \langle f \rangle_R^{\mathrm{sat}}$
by \cite[Lemma~1.6]{Kr2}. Hence we obtain
$I_{\bbY/\bbX} = \langle f \rangle_R^{\mathrm{sat}}$.

(b)\  Obviously, we have
$\alpha_{\bbY/\bbX}\le \mu_{\bbY/\bbX}$
and $\HF_{\bbY}(i) \le \HF_{\bbX}(i)-1$ for
$\alpha_{\bbY/\bbX} \le i <\mu_{\bbY/\bbX}$.
Now we verify the equality
$\HF_{\bbY}(i+\mu_{\bbY/\bbX})=
\HF_{\bbX}(i+\mu_{\bbY/\bbX})-\varkappa_j$
for all $i \ge 0$. We set
$g_{jk_j}:=x_0^{\mu_{\bbY/\bbX}-\deg(f^*_{jk_j})}f^*_{jk_j}
\in (I_{\bbY/\bbX})_{\mu_{\bbY/\bbX}}$
for all $k_j = 1, \dots, \varkappa_j$.
Then we have
$\tilde{\imath}(g_{jk_j})
=(0,\dots,0,e_{jk_j}s_jT_j^{\mu_{\bbY/\bbX}},0,\dots,0)$.
Since $\{e_{j1}s_{j},\dots,e_{j\varkappa_j}s_{j}\}$ is
$K$-linearly independent, this implies
$$
\varkappa_j
= \dim_K \langle g_{j1},\dots,g_{j\varkappa_j}\rangle_K
\le \dim_K(I_{\bbY/\bbX})_{\mu_{\bbY/\bbX}}
\le \varkappa_j.
$$
So, we get $\dim_K(I_{\bbY/\bbX})_{\mu_{\bbY/\bbX}}
=\dim_K(I_{\bbY/\bbX})_{i+\mu_{\bbY/\bbX}}
=\varkappa_j$ for all $i \ge 0$.
It follows that $\HF_{\bbY}(i+\mu_{\bbY/\bbX})
=\HF_{\bbX}(i+\mu_{\bbY/\bbX})-\varkappa_j$
for all $i \ge 0$. In particular, $\mu_{\bbY/\bbX}$
is the smallest number $i\in\bbN$ such that
$\HF_{\bbY}(i)=\HF_{\bbX}(i)-\varkappa_j$.

Moreover, we see that $\HF_\bbY(r_\bbX) = \deg(\bbY)$,
since otherwise we would have
$$
\HF_{I_{\bbY/\bbX}}(r_\bbX)
= \deg(\bbX) - \HF_\bbY(r_\bbX) >
\deg(\bbX)-\deg(\bbY)=\varkappa_j,
$$
which is impossible.
Thus $\HF_{\bbY}(r_{\bbX}) = \deg(\bbX)-\varkappa_j
= \HF_\bbX(r_\bbX)-\varkappa_j$, and hence the inequality
$\mu_{\bbY/\bbX} \le r_{\bbX}$ holds true.

(c)\ We may construct the set
$\{e_{j1},\dots,e_{j\varkappa_j}\}\subseteq\calO_{\bbX,p_j}$
with the desired properties as follows. Let
$d_{\alpha_{\bbY/\bbX}}=\HF_{I_{\bbY/\bbX}}(\alpha_{\bbY/\bbX})$
and
$$
d_{\alpha_{\bbY/\bbX}+i}
=\HF_{I_{\bbY/\bbX}}(\alpha_{\bbY/\bbX}+i)-
\HF_{I_{\bbY/\bbX}}(\alpha_{\bbY/\bbX}+i-1)
$$
for $i=1,\dots,\mu_{\bbY/\bbX}-\alpha_{\bbY/\bbX}$.
Then we have
$\varkappa_j=d_{\alpha_{\bbY/\bbX}}
+d_{\alpha_{\bbY/\bbX}+1}+\cdots+d_{\mu_{\bbY/\bbX}}$.
We begin taking a $K$-basis
$f^*_{j1},\dots,f^*_{jd_{\alpha_{\bbY/\bbX}}}$
of $(I_{\bbY/\bbX})_{\alpha_{\bbY/\bbX}}$.
For $i=1,\dots,\mu_{\bbY/\bbX}-\alpha_{\bbY/\bbX}$, if
$d_{\alpha_{\bbY/\bbX}+i}>0$, we choose
$f^*_{j{\scriptscriptstyle\sum\limits_{0\leq k<i}
d_{\alpha_{\bbY/\bbX}+k}}+1},\dots,
f^*_{j{\scriptscriptstyle\sum\limits_{ 0\leq k\leq i}
d_{\alpha_{\bbY/\bbX}+k}}}$
such that the set
$$
\begin{aligned}
\Big\{x_0^{i}f^*_{j1},\dots,x_0^{i}f^*_{jd_{\alpha_{\bbY/\bbX}}},
\dots,
& x_0f^*_{j{\scriptscriptstyle\sum\limits_{0\le k<i-1}
  d_{\alpha_{\bbY/\bbX}+k}}+1},\dots,
  x_0f^*_{j{\scriptscriptstyle\sum\limits_{0\le k\le i-1}
  d_{\alpha_{\bbY/\bbX}+k}}},\\
& \hspace*{1.55cm} f^*_{j{\scriptscriptstyle\sum\limits_{0\le k<i}
  d_{\alpha_{\bbY/\bbX}+k}}+1},\dots,
  f^*_{j{\scriptscriptstyle\sum\limits_{0\leq k\le i}
  d_{\alpha_{\bbY/\bbX}+k}}}\Big\}
\end{aligned}
$$
forms a $K$-basis of $(I_{\bbY/\bbX})_{\alpha_{\bbY/\bbX}+i}$.
Then the ideal
$J=\langle f^*_{j1}, \dots, f^*_{j\varkappa_j}\rangle_R$
is a subideal of~$I_{\bbY/\bbX}$ and
$\HF_{J}(i) = \HF_{I_{\bbY/\bbX}}(i)$ for all
$i\le \mu_{\bbY/\bbX}$.
By~(b) we have $\HF_{J}(i)=\HF_{I_{\bbY/\bbX}}(i)=\varkappa_j$
for $i \ge \mu_{\bbY/\bbX}$. This implies $I_{\bbY/\bbX} = J
= \langle f^*_{j1}, \dots, f^*_{j\varkappa_j} \rangle_R$.
Moreover, it follows from the construction of the set
$\{ f^*_{j1}, \dots, f^*_{j\varkappa_j} \}$ that
$$
\HF_{I_{\bbY/\bbX}}(i)=\#\big\{\,
k \in\{1,\dots,\varkappa_j\}
\mid \deg(f^*_{jk}) \le i \,\big\}
$$
for all $i \in \bbZ$.
Thus we have
$$
\begin{aligned}
\Delta\HF_{\bbY}(i)& = \HF_{\bbY}(i) - \HF_{\bbY}(i-1)\\
&=(\HF_{\bbX}(i) - \HF_{I_{\bbY/\bbX}}(i)) -
(\HF_{\bbX}(i-1) - \HF_{I_{\bbY/\bbX}}(i-1))\\
&=\Delta\HF_{\bbX}(i) - (\HF_{I_{\bbY/\bbX}}(i)
- \HF_{I_{\bbY/\bbX}}(i-1))\\
&=\Delta\HF_{\bbX}(i)-
\#\big\{\, k \in\{1,\dots,\varkappa_j\}
\mid \deg(f^*_{jk}) = i\,\big\}.
\end{aligned}
$$

Now let us write $\tilde{\imath}(f^*_{jk_j})
=(0,\dots,0,e_{jk_j}s_jT_j^{\deg(f^*_{jk_j})},0,\dots,0)$
for $k_j = 1, \dots, \varkappa_j$. Obviously, the set
$\{ e_{j1}s_j, \dots, e_{j\varkappa_j}s_j \}$
is $K$-linearly independent.
It remains to show that the residue classes
$\{ \overline{e}_{j1}, \dots, \overline{e}_{j\varkappa_j} \}$
form a $K$-basis of~$\kappa(p_j)$.
Suppose there are $c_{j1}, \dots, c_{j\varkappa_j} \in K$
such that
$c_{j1}\overline{e}_{j1} + \cdots +
c_{j\varkappa_j}\overline{e}_{j\varkappa_j} = 0$.
It follows that the element
$c_{j1}e_{j1} + \cdots + c_{j\varkappa_j}e_{j\varkappa_j}$
is contained in $\fm_{\bbX,p_j}$.
This implies  $c_{j1}e_{j1}s_j + \cdots +
c_{j\varkappa_j}e_{j\varkappa_j}s_j = 0$.
Since $\{ e_{j1}s_j, \dots, e_{j\varkappa_j}s_j \}$
is $K$-linearly independent, we deduce
$c_{j1} = \cdots = c_{j\varkappa_j}= 0$.
Therefore the set
$\{ \overline{e}_{j1}, \dots, \overline{e}_{j\varkappa_j} \}$
is a $K$-basis of~$\kappa(p_j)$, and the conclusion follows.
\end{proof}

The set of minimal separators
$\{f^*_{j1},\dots,f^*_{j\varkappa_j}\}$
of a maximal $p_j$-subscheme $\bbY$ in~$\bbX$ as
in Proposition~\ref{propSec1.5}(c)
is not necessarily a homogeneous minimal system of
generators of~$I_{\bbY/\bbX}$, as the following example shows.

\begin{exam}
Let $\bbX \subseteq \bbP^2_{\bbQ}$ be the $0$-dimensional
reduced complete intersection
with $I_{\bbX} = \langle X_{2},
X_{0}^{5}X_{1} -\frac{11}{6}X_{0}^{4}X_{1}^{2}
+ 2X_{0}^{3}X_{1}^{3} - 2X_{0}^{2}X_{1}^{4}
+ X_{0}X_{1}^{5} -\frac{1}{6}X_{1}^{6}\rangle$.
Then~$\bbX$ contains the set of $\bbQ$-rational points
$\bbY = \{ (1:0:0), (1:1:0), (1:2:0), (1:3:0) \}$
which is a maximal $p$-subscheme, where $p$ is the closed
point corresponding to the homogeneous prime ideal
$\fP = \langle X_1^2+X_0^2, X_2 \rangle$.
We see that $\deg(\bbY) = 4 = \deg(\bbX) - 2$, and
two minimal separators of~$\bbY$ in~$\bbX$
are $f^*_{1} = x_{0}^{3}x_{1} -\frac{11}{6}x_{0}^{2}x_{1}^{2}
+ x_{0}x_{1}^{3} -\frac{1}{6}x_{1}^{4}$
and $f^*_{2} = x_1f^*_{1}$.
Moreover, the equality of the first difference function
of~$\HF_\bbY$ in Proposition~\ref{propSec1.5}(c) holds true,
while $I_{\bbY/\bbX} = \langle f^*_1 \rangle_R$.
\end{exam}

When $\bbY\subseteq\bbX$ is a $p_j$-subscheme of degree
$\deg(\bbY)=\deg(\bbX)-1$, we have
$\alpha_{\bbY/\bbX}=\mu_{\bbY/\bbX}$
and the Hilbert function of~$\bbY$ is given by
$$
\HF_{\bbY}(i)=
\begin{cases}
\HF_{\bbX}(i)   &\ \textrm{for}\ i< \alpha_{\bbY/\bbX},\\
\HF_{\bbX}(i)-1 &\ \textrm{for} \ i\ge \alpha_{\bbY/\bbX}
\end{cases}
$$
(see also \cite[Lemma~1.7]{Kr2}).
Furthermore, if $\bbX= \{p_1,\dots,p_s\}$ is a set of
distinct $K$-rational points in~$\bbP^n_K$, we write
$p_j=(1: p_{j1}:...:p_{jn})$ with $p_{jk}\in K$,
and for $f\in R$ we set $f(p_j):=F(1, p_{j1},\dots,p_{jn})$
where $F$ is any representative of~$f$ in~$P$.
Then a separator of $\bbX\setminus\{p_j\}$ in~$\bbX$
is an element $f\in R_{r_{\bbX}}$ such that $f(p_j) \ne 0$ and
$f(p_k) = 0$ for $k \ne j$. In general setting we introduce
the following definition.

\begin{defn}
In the setting of Proposition~\ref{propSec1.5},
we let $f_{jk_j}=x_0^{r_{\bbX}-\mu(e_{jk_j}s_{j})}f^*_{jk_j}$
for $k_j = 1, \dots, \varkappa_j$.
The set $\{f_{j1},\dots,f_{j\varkappa_j}\}$  is called the
{\bf standard set of separators of $\bbY$ in~$\bbX$}
with respect to $s_j$ and $\{e_{j1},\dots,e_{j\varkappa_j}\}$.
\end{defn}

Some basic properties of standard sets of separators of a maximal
$p_j$-subscheme are summarized in the following lemma
which generalizes some results in~\cite[Lemmas~1.9 and 1.10]{Kr2}.

\begin{lem}\label{lemSec1.8}
Let $\bbX \subseteq \bbP^n_K$ be a $0$-dimensional scheme,
let $f \in R_i$ with $i \ge 0$, let $\bbY$ be a maximal
$p_j$-subscheme of~$\bbX$, and let
$\{ f_{j1},\dots,f_{j\varkappa_j}\}\subseteq R_{r_{\bbX}}$
be a standard set of separators of~$\bbY$ in~$\bbX$.
\begin{enumerate}
  \item We have $f \cdot f_{jl} =
   \sum_{k_j=1}^{\varkappa_j}c_{jk_jl}x_0^if_{jk_j}$
   for some $c_{j1l},\dots,c_{j\varkappa_jl} \in K$ and
   $l \in \{ 1, \dots, \varkappa_j \}$.

  \item If $f \cdot f_{jl} = 0$ for some
   $l \in \{1,\dots,\varkappa_j\}$, then
   $f \cdot f_{j\lambda} = 0$ for all
   $\lambda \in \{ 1, \dots, \varkappa_j \}$.
   Moreover, $f \cdot f_{jl} \ne 0$  if and only if
   $f_{p_j} \notin \fm_{\bbX,p_j}$.

  \item Let $\bbY'$ be a maximal $p_{j'}$-subscheme of~$\bbX$,
   and let $\{ f_{j'1}, \dots, f_{j'\varkappa_{j'}} \}
   \subseteq R_{r_{\bbX}}$ be a standard set of
   separators of~$\bbY'$ in~$\bbX$. Then we have
   $$
   \quad f_{jk_j} \cdot f_{j'k_{j'}}\in
   \begin{cases}
   x_0^{r_{\bbX}}\langle f_{j1},\dots,f_{j\varkappa_j}\rangle_K
   &\ \textrm{if}\  j = j' \ \textrm{and} \
   \dim_{\kappa(p_j)}(\calO_{\bbX,p_j}) = 1,\\
   \langle 0 \rangle   &\ \textrm{otherwise}.
   \end{cases}
   $$
\end{enumerate}
\end{lem}
\begin{proof}
Claim (a) is a consequence of the fact that
$$
f \cdot f_{jl} \in (I_{\bbY/\bbX})_{r_{\bbX}+i}
=\langle x_0^if_{j1}, \dots, x_0^if_{j\varkappa_j} \rangle_K
$$
for $l = 1, \dots, \varkappa_j$.
Claims~(b) and~(c) follow by using the injection
$\tilde{\imath}$ and the fact that $(f_{jk_j})_{p_j}$
is a socle element of~$\fG(\calO_{\bbX,p_j})$ for
$k_j = 1, \dots, \varkappa_j$.
\end{proof}

In the case that the scheme $\bbX$ is reduced and
$i\ge r_\bbX$, we can use standard sets of separators of~$\bbX$
to describe a $K$-basis of the vector space $R_i$
as follows (see \cite[Proposition~1.13(a)]{GKR} for the case
of sets of distinct $K$-rational points).

\begin{cor}
Let \,$\bbX \,\subseteq\, \bbP^n_K$\, be a reduced
$0$-dimensional scheme with support
$\Supp(\bbX)=\{ p_1, \dots, p_s \}$, let
$\{ f_{j1}, \dots, f_{j\varkappa_j} \}
\subseteq R_{r_{\bbX}}$ be a standard set of separators
of~$\bbX\setminus\{p_j\}$ in~$\bbX$ for $j = 1, \dots, s$.
Then the set
$$
\{\, x_0^{i-r_{\bbX}}f_{11}, \dots,
   x_0^{i-r_{\bbX}}f_{1\varkappa_1}, \dots,
   x_0^{i-r_{\bbX}}f_{s1}, \dots,
   x_0^{i-r_{\bbX}}f_{s\varkappa_s} \,\}
$$
is a $K$-basis of~$R_i$ for every $i \ge r_{\bbX}$.
\end{cor}

\begin{proof}
Since the scheme $\bbX$ is reduced, we have
$\calO_{\bbX,p_j} = \kappa(p_j) = \fG(\calO_{\bbX,p_j})$
for $j = 1, \dots, s$. Let $i \ge r_{\bbX}$. We write
$$
\tilde{\imath}(x_0^{i-r_{\bbX}}f_{jk_j})
=(0, \dots, 0, e_{jk_j}T_j^{i}, 0, \dots, 0)
\in Q^h(R)
$$
for $j = 1, \dots, s$ and $k_j = 1, \dots, \varkappa_j$,
where $\{e_{j1}, \dots, e_{j\varkappa_j}\}$ is
a $K$-basis of~$\calO_{\bbX,p_j}$. Then the set
$\{\, \tilde{\imath}(x_0^{i-r_{\bbX}}f_{11}),
\dots, \tilde{\imath}(x_0^{i-r_{\bbX}}f_{s\varkappa_s})\,\}$
is $K$-linearly independent, and so it forms a $K$-basis
of~$Q^h(R)_i$. Since $i \ge r_{\bbX}$, the restriction
$\tilde{\imath}|_{R_{i}}: R_{i}\rightarrow Q^h(R)_i$
is an isomorphism of $K$-vector spaces, it follows that
$\{\, x_0^{i-r_{\bbX}}f_{11}, \dots,
x_0^{i-r_{\bbX}}f_{s\varkappa_s} \,\}$
is a $K$-basis of~$R_i$.
\end{proof}

\medskip\bigbreak
\section{Dedekind Differents of 0-Dimensional Schemes}

In this section we define and examine the Dedekind complementary
module and the Dedekind different for a 0-dimensional scheme
$\bbX\subseteq\bbP^n_K$. For this we need to restrict our
attention to a special class of $0$-dimensional schemes,
namely locally Gorenstein schemes.
Here we say that $\bbX$ is {\bf locally Gorenstein}
if the local ring $\calO_{\bbX,p_j}$ is a Gorenstein ring
for every point $p_j\in\Supp(\bbX)$.

Recall that the graded $R$-module
$\omega_R = \underline{\Hom}_{K[x_0]}(R,K[x_0])(-1)$
is called the {\bf canonical module} of~$R$.
It is a finitely generated graded $R$-module
with Hilbert function
$\HF_{\omega_R}(i)=\deg(\bbX)-\HF_\bbX(-i)$
for all $i\in\bbZ$ (see \cite[Proposition~1.3]{Kr1}).

In the following we assume that $\bbX\subseteq\bbP^n_K$ is
a 0-dimensional locally Gorenstein scheme and let
$L_0 = K[x_0,x_0^{-1}]$.
In this case one can embed the canonical module $\omega_R$
of~$R$ as a fractional ideal into its homogeneous ring
of quotients $Q^h(R)$ (see \cite{HKW} or \cite[Appendix~G]{Ku2}).
Explicitly, this construction is based on the existence of
a homogeneous trace map of the graded algebra $Q^h(R)/L_0$.
Recall that a {\bf homogeneous trace map} of
a finite graded algebra~$T/S$ is a homogeneous $T$-basis of
the graded module $\underline{\Hom}_{S}(T,S)$.
For further information on (canonical, homogeneous) trace maps
we refer to \cite[Appendix~F]{Ku2}.

The following proposition indicates that the graded algebra
$Q^h(R)/L_0$ has a homogeneous trace map of degree zero,
which is shown in~\cite[Proposition~3.3]{KL}.

\begin{prop} \label{propSec2.1}
The following statements hold true.
\begin{enumerate}
\item The algebra $Q^h(R)/L_0$ has a homogeneous trace
map $\sigma$ of degree zero.

\item The map
$\Sigma: Q^h(R)\rightarrow \underline{\Hom}_{L_0}(Q^h(R),L_0)$
given by $\Sigma(1)=\sigma$ is an isomorphism of graded
$Q^h(R)$-modules.
\end{enumerate}
\end{prop}

Now let $\sigma$ be a fixed homogeneous trace map of degree zero
of~$Q^h(R)/L_0$. Note that
$\sigma\in \underline{\Hom}_{L_0}(Q^h(R),L_0)$
satisfies $\underline{\Hom}_{L_0}(Q^h(R),L_0)
= Q^h(R)\cdot\sigma$.
Furthermore, there is an injective homomorphism of
graded $R$-modules
\begin{equation}\label{mapSec2.1}
\begin{aligned}
\Phi: \omega_R(1) &
\lhook\joinrel\longrightarrow \underline{\Hom}_{L_0}(Q^h(R),L_0)
= Q^h(R)\cdot\sigma \xrightarrow{\Sigma^{-1}} Q^h(R)\\
\varphi &\longmapsto \varphi\otimes \mathrm{id}_{L_0}
\end{aligned}
\end{equation}
The image of~$\Phi$ is a homogeneous fractional
$R$-ideal~$\fC^{\sigma}_\bbX$ of~$Q^h(R)$. It is also
a finitely generated graded $R$-module and
$$
\HF_{\fC^{\sigma}_{\bbX}}(i) = \deg(\bbX) - \HF_{\bbX}(-i-1)
\quad \textrm{for all} \ i\in\bbZ.
$$

\begin{defn}
The $R$-module $\fC^{\sigma}_{\bbX}$
is called the \textbf{Dedekind complementary module} of~$\bbX$
(or of~$R/K[x_0]$) with respect to~$\sigma$. Its inverse,
$$
\delta^{\sigma}_\bbX = (\fC^{\sigma}_\bbX)^{-1}
= \{\, f \in Q^h(R) \, \mid \,
f \cdot \fC^{\sigma}_\bbX \subseteq R \, \},
$$
is called the \textbf{Dedekind different} of $\bbX$
(or of~$R/K[x_0]$) with respect to~$\sigma$.
\end{defn}

When $\bbX$ is a finite set of distinct $K$-rational points
of~$\bbP^n_K$, we also denote the Dedekind complementary
module (respectively, the Dedekind different)
with respect to the canonical trace map by $\fC_\bbX$
(respectively, $\delta_\bbX$).

A system of generators of $\fC^{\sigma}_{\bbX}$ can be
computed as follows.

\begin{rem}\label{remSec2.3}
Let $<_{\tau}$ be a degree-compatible term ordering
on the set of terms $\mathbb{T}^n$ of~$K[X_1,\dots,X_n]$,
and let $d = \deg(\bbX)$. Then
$\mathbb{T}^n\setminus \mathrm{LT}_{\tau}(I_{\bbX}^{\mathrm{deh}})
= \{\, T'_1, \dots, T'_d \,\}$ with
$T'_j = X_1^{\alpha_{j1}}\cdots X_n^{\alpha_{jn}}$ and
$\alpha_j=(\alpha_{j1}, \dots, \alpha_{jn}) \in \bbN^n$
for $j = 1, \dots, d$. W.l.o.g. we assume that
$T'_1<_{\tau}\cdots<_{\tau}T'_m$.
Let $t_j = T'_j + I_{\bbX} \in R$ and set
$\deg(t_j) := \deg(T'_j) = n_j$ for $j = 1, \dots, d$.
Then $n_1 \le \cdots \le n_d \le r_{\bbX}$ and the set
$\{\, t_1, \dots, t_d \,\}$ is a $K[x_0]$-basis of~$R$
(cf.~\cite[Theorem~4.3.22]{KR2}).
Let $\{t_1^*, \dots, t_d^*\}$ be the dual basis
of~$\{t_1, \dots, t_d\}$,
and let $g_j = \Phi(t_j^*)$ for $j = 1, \dots, d$.
We get $\fC^{\sigma}_{\bbX} = \langle g_1, \dots, g_d
\rangle_{K[x_0]} \subseteq Q^h(R)$.
\end{rem}

Now we want to take a closer look at each homogeneous component
of the Dedekind complementary module of $\bbX$. For this we use
the following notation.  Let $\nu_j := \dim_K(\calO_{\bbX,p_j})$
and let $\{e_{j1}, \dots, e_{j\nu_j}\}$ be a $K$-basis
of~$\calO_{\bbX,p_j}$ for $j=1,\dots,s$.
Using the injection
$\tilde{\imath}: R \hookrightarrow Q^h(R)$,
we set
$$
f_{jk_j} :=
\tilde{\imath}^{-1}((0,\dots,0,e_{jk_j}T_j^{r_{\bbX}},0,\dots,0))
$$
for $k_j = 1, \dots, \nu_j$. It is easy to see that
$R_{r_{\bbX}} = \langle f_{11},\dots,f_{1\nu_1},\dots,
f_{s1},\dots,f_{s\nu_s} \rangle_K$.
Since $\bbX$ is locally Gorenstein, $\calO_{\bbX,p_j}/K$ has
a trace map
$\overline{\sigma}_j \in \Hom_{K}(\calO_{\bbX,p_j},K)$.
Also, there is a $K$-basis  $\{e'_{j1}, \dots, e'_{j\nu_j}\}$
of~$\calO_{\bbX,p_j}$ such that
$$
\overline{\sigma}_j(e_{jk_j}e'_{jk'_j}) =
e^*_{jk'_j}(e_{jk_j}) = \delta_{k_jk'_j}
$$
for all $k_j, k'_j = 1, \dots, \nu_j$.
The $K$-basis $\{e'_{j1}, \dots, e'_{j\nu_j}\}$ is known as a
{\bf dual basis of $\calO_{\bbX,p_j}$ to the $K$-basis
$\{e_{j1},\dots,e_{j\nu_j}\}$ w.r.t.~$\overline{\sigma}_j$}.
Moreover, these maps $\overline{\sigma}_j$ induce
a homogeneous trace map $\sigma$ of degree zero of~$Q^h(R)/L_0$.

A description of the Dedekind complementary module of $\bbX$
is given by our next proposition.

\begin{prop}\label{propSec3.4}
Using the above notation, let $\Phi$ be the monomorphism
of graded $R$-modules defined by~(\ref{mapSec2.1}),
let $i \ge 0$, and let $\varphi \in (\omega_R)_{i-r_{\bbX}+1}$.
We write $\varphi(f_{jk_j}) = c_{jk_j}x_0^i$ with
$c_{jk_j} \in K$. Then we have
$$
\Phi(\varphi) =
\big({\textstyle\sum\limits_{k_1=1}^{\nu_1}}
c_{1k_1}e'_{1k_1}T_1^{i-r_{\bbX}}, \dots,
{\textstyle\sum\limits_{k_s=1}^{\nu_s}}
c_{sk_s}e'_{sk_s}T_s^{i-r_{\bbX}} \big)
\in (\fC^{\sigma}_\bbX)_{i-r_{\bbX}}.
$$
In particular, $\Phi(\varphi)$ can be identified with the element
$x_0^{i-2r_{\bbX}}(\sum_{j=1}^s\sum_{k_j=1}^{\nu_j}
c_{jk_j}\widetilde{f}_{jk_j})$ of $R_{x_0}\cong Q^h(R)$,
where $\widetilde{f}_{jk_j}=
\tilde{\imath}^{-1}((0,\dots,0,e'_{jk_j}T_j^{r_{\bbX}},0,\dots,0))
\in R_{r_{\bbX}}$ for all $j=1,\dots,s$ and for all
$k_j = 1, \dots, \nu_j$.
\end{prop}

\begin{proof}
We set $\epsilon_{jk_j}:=(0, \dots, 0, e_{jk_j}, 0, \dots, 0)
\in \prod_{l=1}^s \calO_{\bbX,p_l}$ for $j = 1, \dots, s$
and $k_j = 1, \dots, \nu_j$.
It is not difficult to see that the set
$\{\epsilon_{11}, \dots, \epsilon_{1\nu_1}, \dots,
\epsilon_{s1}, \dots, \epsilon_{s\nu_s}\}$
is a $L_0$-basis of~$Q^h(R)$.
So, the mapping $\varphi\otimes\mathrm{id}_{L_0}: Q^h(R)
\cong  R\otimes_{K[x_0]}L_0 \rightarrow L_0$ satisfies
$$
(\varphi\otimes\mathrm{id}_{L_0})(x_0^{r_{\bbX}}\epsilon_{jk_j})
= (\varphi\otimes\mathrm{id}_{L_0})((0, \dots, 0,
e_{jk_j}T_j^{r_{\bbX}}, 0, \dots, 0))
= \varphi(f_{jk_j}) = c_{jk_j}x_0^i
$$
for $j = 1, \dots, s$ and $k_j = 1, \dots, \nu_j$.
Thus we have $(\varphi\otimes\mathrm{id}_{L_0})(\epsilon_{jk_j})
= c_{jk_j}x_0^{i-r_{\bbX}}$
for all $j = 1, \dots, s$ and $k_j = 1, \dots, \nu_j$.
On the other hand, we see that
$$
\begin{aligned}
&\big( {\textstyle\sum\limits_{k_1=1}^{\nu_1}} c_{1k_1}e'_{1k_1},
\dots, {\textstyle\sum\limits_{k_s=1}^{\nu_s}}
c_{sk_s}e'_{sk_s} \big) \cdot \sigma(\epsilon_{jk_j})
=\sigma\big((0,\dots,0,{\textstyle\sum\limits_{k'_j=1}^{\nu_j}}
c_{jk'_j}e'_{jk'_j}e_{jk_j},0,\dots,0)\big)\\
&=\overline{\sigma}_j\big({\textstyle\sum\limits_{k'_j=1}^{\nu_j}}
c_{jk'_j}e'_{jk'_j}e_{jk_j} \big)
={\textstyle\sum\limits_{k'_j=1}^{\nu_j}}
c_{jk'_j}\overline{\sigma}_j(e'_{jk'_j}e_{jk_j})
= {\textstyle\sum\limits_{k'_j=1}^{\nu_j}}
c_{jk'_j}\delta_{k_jk'_j} =c_{jk_j}.
\end{aligned}
$$
This implies that we have
$$
(\varphi\otimes\mathrm{id}_{L_0}) = x_0^{i-r_{\bbX}} \big(
{\textstyle\sum\limits_{k_1=1}^{\nu_1}} c_{1k_1}e'_{1k_1},\dots,
{\textstyle\sum\limits_{k_s=1}^{\nu_s}} c_{sk_s}e'_{sk_s} \big)
\cdot \sigma
$$
in $\underline{\Hom}_{L_0}(Q^h(R),L_0)$.
Hence we get
$$
\Phi(\varphi)=\big( {\textstyle\sum\limits_{k_1=1}^{\nu_1}}
c_{1k_1}e'_{1k_1}T_1^{i-r_{\bbX}}, \dots,
{\textstyle\sum\limits_{k_s=1}^{\nu_s}}
c_{sk_s}e'_{sk_s}T_s^{i-r_{\bbX}} \big)
\in (\fC^{\sigma}_\bbX)_{i-r_{\bbX}}.
$$

In addition, we observe that
$$
\begin{aligned}
x_0^{2r_{\bbX}}\Phi(\varphi)&= x_0^{2r_{\bbX}} \cdot
\big({\textstyle\sum\limits_{k_1=1}^{\nu_1}}
c_{1k_1}e'_{1k_1}T_1^{i-r_{\bbX}},
\dots, {\textstyle\sum\limits_{k_s=1}^{\nu_s}}
c_{sk_s}e'_{sk_s}T_s^{i-r_{\bbX}} \big)\\
&= \big( {\textstyle\sum\limits_{k_1=1}^{\nu_1}}
c_{1k_1}e'_{1k_1}T_1^{r_{\bbX}+i},\dots,
{\textstyle\sum\limits_{k_s=1}^{\nu_s}}
c_{sk_s}e'_{sk_s}T_s^{r_{\bbX}+i} \big) \\
&= \tilde{\imath} \big( x_0^i ({\textstyle\sum\limits_{j=1}^s
\sum\limits_{k_j=1}^{\nu_j}} c_{jk_j}\widetilde{f}_{jk_j}) \big).
\end{aligned}
$$
Therefore the claim follows.
\end{proof}

Next we collect from~\cite[Proposition~3.7]{KL} the following
basic properties of the Dedekind different of $\bbX$.

\begin{prop}\label{propSec2.5}
Let~$\sigma$ be a trace map of $Q^h(R)/L_0$.

\begin{enumerate}
  \item The Dedekind different $\delta^{\sigma}_\bbX$ is
  a homogeneous ideal of~$R$ and $x_0^{2r_{\bbX}}
  \in \delta^{\sigma}_\bbX$.

  \item The Hilbert function of~$\delta^{\sigma}_\bbX$ satisfies
   $\HF_{\delta^{\sigma}_\bbX}(i) = 0$ for $i < 0$,
   $\HF_{\delta^{\sigma}_\bbX}(i) = \deg(\bbX)$ for
   $i \ge 2r_{\bbX}$, and
   $$
   0 \le \HF_{\delta^{\sigma}_\bbX}(0) \le \cdots
     \le \HF_{\delta^{\sigma}_\bbX}(2r_{\bbX}) = \deg(\bbX).
   $$

  \item The regularity index of~$\delta^{\sigma}_\bbX$
   satisfies $r_{\bbX} \le \ri(\delta^{\sigma}_\bbX)
   \le 2r_{\bbX}$.
\end{enumerate}
\end{prop}

The upper bound for the regularity index of the Dedekind different
given in this proposition is attained for a finite
set of distinct $K$-rational points, as the next corollary shows.

\begin{cor}
Let $\bbX=\{p_1,\dots,p_s\}\subseteq\bbP^n_K$ be
a set of $s$ distinct $K$-rational points. Then we have
$\HP_{\delta_\bbX}(z) = s$ and $\ri(\delta_\bbX) = 2r_{\bbX}$.
\end{cor}

\begin{proof}
If $n=1$, then $\bbX$ is a complete intersection,
and so $\HF_{\delta_\bbX}(i) = \HF_{\bbX}(i-s+1)$ for all
$i\in\bbZ$. In particular, we have
$\ri(\delta_\bbX) = 2r_{\bbX} = (n+1)(s-1)$.

Now suppose that $n\ge 2$.
For $j\in\{1,\dots,s\}$, let $f_j\in R_{r_\bbX}$ be
the separator of~$\bbX\setminus\{p_j\}$ in~$\bbX$ with
$f(p_j)=1$ and $f(p_k)=0$ for $k\ne j$, and let
$\overline{f}_j$ denote the image of $f_j$
in~$\overline{R}:=R/\langle x_0\rangle$.
Set $\Delta_\bbX := \dim_K \overline{R}_{r_{\bbX}}
= \HF_\bbX(r_\bbX)-\HF_\bbX(r_\bbX-1)$.
Note that $\Delta_\bbX \ge 1$.
Since $\{\overline{f}_1,\dots,\overline{f}_s\}$ generates
the $K$-vector space $\overline{R}_{r_{\bbX}}$,
we can renumber $\{p_1,\dots,p_s\}$ in such a way that
$\{\overline{f}_1,\dots,\overline{f}_{\Delta_\bbX}\}$
is a $K$-basis of $\overline{R}_{r_{\bbX}}$.
Because $\overline{f}_i\ne 0$ for every $i\in\{1,\dots,\Delta_\bbX\}$,
this implies $f_1,\dots,f_{\Delta_\bbX}\notin x_0R_{r_\bbX-1}$.
For $j=1,\dots,s-\Delta_\bbX$, we write
$$
\overline{f}_{\Delta_\bbX+j} =
\beta_{j1}\overline{f}_1+\cdots+
\beta_{j\Delta_\bbX}\overline{f}_{\Delta_\bbX}
$$
where $\beta_{j1},\dots,\beta_{j\Delta_\bbX} \in K$.
By~\cite[Corollary~1.10]{Kr3}, the elements
$$
\widetilde{g}_j= x_0^{-2r_{\bbX}}(f_j +
\beta_{1j}f_{\Delta_\bbX+1}+\cdots+\beta_{s-\Delta_\bbX\; j}f_{s})
$$
such that $1 \le j \le \Delta_\bbX$ form a $K$-basis
of~$(\fC_\bbX)_{-r_{\bbX}}$.

Now suppose for a contradiction that
$\HF_{\delta_\bbX}(2r_\bbX-1) = s$.
This implies that $(\delta_\bbX)_{2r_\bbX-1}
= R_{2r_\bbX-1}$. In particular, we have
$x_0^{r_\bbX-1}f_1 \in (\delta_\bbX)_{2r_\bbX-1}$.
Using Lemma~\ref{lemSec1.8}, we also have
$$
\begin{aligned}
x_0^{r_\bbX-1}f_1\cdot \widetilde{g}_1
&= x_0^{r_\bbX-1}f_1\cdot x_0^{-2r_{\bbX}}
(f_1+\beta_{11}f_{\Delta_\bbX+1}+\cdots
+\beta_{s-\Delta_\bbX\; 1}f_{s}) \\
&= x_0^{-r_{\bbX}-1}f_1^2 = x_0^{-1}f_1 \in R_{r_\bbX-1}.
\end{aligned}
$$
It follows that $f_1 \in x_0R_{r_\bbX-1}$, a contradiction.
Thus we must have $\HF_{\delta_\bbX}(2r_\bbX-1) < s$,
and hence $\ri(\delta_\bbX) = 2r_{\bbX}$.
\end{proof}

In view of the preceding proposition, for a 0-dimensional
locally Gorenstein scheme $\bbX$ the inclusion
$\bigoplus_{i\ge 2r_\bbX} R_i \subseteq \delta_\bbX^\sigma$
always holds true. When this inclusion becomes an equality,
we use the following name.

\begin{defn}
We say that $\bbX$ has {\bf minimal Dedekind different}
if its Dedekind different satisfies
$\delta_\bbX^\sigma=\bigoplus_{i\ge 2r_\bbX} R_i$.
\end{defn}

Recall that the K\"{a}hler different $\vartheta_\bbX$
of~$\bbX$ is the homogeneous ideal of~$R$ generated
by all $n$-minors of the Jacobian matrix
$\left(\frac{\partial F_{j}}{\partial x_i}\right)_{
\begin{subarray}{l}
i=1,\dots,n\\ j=1,\dots,r
\end{subarray}}$,
where $\{F_1,\dots,F_r\}$ is a homogeneous system of
generators of~$I_\bbX$.
For finite sets of distinct $K$-rational points
in~$\bbP^2_K$ which have minimal Dedekind different,
the Dedekind and K\"{a}hler differents agree,
as the following corollary shows.

\begin{cor}\label{corSec2.8}
Let $\bbX=\{p_1,\dots,p_s\}\subseteq\bbP^2_K$ be
a set of $s$ distinct $K$-rational points.
If $\bbX$ has minimal Dedekind different, then
$\delta_\bbX = \vartheta_\bbX$.
\end{cor}

\begin{proof}
By \cite[Proposition~3.8]{KL}, we have
$\vartheta_\bbX\subseteq \delta_\bbX$.
Because $\bbX$ has minimal Dedekind different,
we have $\HP_{\vartheta_\bbX}(2r_\bbX-1) =
\HF_{\delta_\bbX}(2r_\bbX-1) =0$.
Moreover, it follows from \cite[Theorem~2.5]{KLL}
and $n=2$ that $\ri(\vartheta_\bbX)\le nr_\bbX = 2r_\bbX$
and $\HF_{\vartheta_\bbX}(i) = s$ for all $i\ge 2r_\bbX$.
Thus we obtain
$\delta_\bbX = \vartheta_\bbX = \bigoplus_{i\ge 2r_\bbX} R_i$.
\end{proof}

\begin{exam}\label{examSec2.9}
Let $\bbX=\{p_1, \dots, p_6\}\subseteq \bbP^2_{\bbQ}$
be the set of six points given by $p_1=(1:0:0)$,
$p_2=(1:2:0)$, $p_3=(1:2:1)$, $p_4=(1:0:2)$,
$p_5=(1:1:2)$, and $p_6=(1:2:2)$.
We sketch $\bbX$ in the affine plane
$D_+(X_0)=\mathbb{A}^2_\bbQ$ as follows:
$$
{\scriptstyle
\begin{array}{ccccc}
(0,2)& \bullet & \bullet & \bullet & (2,2)\\
     &         &         & \bullet &  \\
(0,0)& \bullet &         & \bullet &  (2,0)
\end{array}}
$$
Then $\bbX$ has the Hilbert function
$\HF_{\bbX}:\ 1\ 3\ 6\ 6\cdots$
and the regularity index $r_{\bbX}=2$.
Moreover, the Dedekind different is given by
$$
\delta_\bbX
= \langle x_{2}^{4}, x_{1}x_{2}^{3}, x_{0}x_{2}^{3},
x_{1}^{4}, x_{0}x_{1}^{3}, x_{0}^{4} \rangle
= {\textstyle\bigoplus\limits_{i\geq 4}}R_i.
$$
Thus the scheme $\bbX$ has minimal Dedekind different,
and Corollary~\ref{corSec2.8} yields that
$\delta_\bbX = \vartheta_\bbX = \bigoplus_{i\ge 4} R_i$.
\end{exam}

Notice that the Dedekind and K\"{a}hler differents do not
always agree, e.g. when $\bbX$ is a non-reduced complete
intersection in~$\bbP^2_K$ (see \cite[Example~3.9]{KL}).
However, for finite sets of distinct points in~$\bbP^2_K$
we propose the following conjecture.

\begin{conjecture}\label{conjS2.9}
Let $\bbX=\{p_1,\dots,p_s\}\subseteq\bbP^2_K$ be
a set of $s$ distinct $K$-rational points.
Then we have $\delta_\bbX = \vartheta_\bbX$.
\end{conjecture}

Recall that a 0-dimensional scheme $\bbX\subseteq\bbP^n_K$
is an {\bf almost complete intersection} if
$I_\bbX$ is minimally generated by $n+1$ homogeneous polynomials
in~$P$. The above conjecture holds true when the set
$\bbX$ is an almost complete intersection.
This follows from~\cite[Satz~4]{Wal}, because in this case
the Hilbert-Burch Theorem (cf. \cite[Theorem~24.2]{Pev})
implies that $\bbX$ is also a special almost complete intersection
(see \cite[Definition~1]{Wal}).
Note that Corollary~\ref{corSec2.8} and Conjecture~\ref{conjS2.9}
are not true in~$\bbP^3_K$.

\begin{exam}\label{examSec2.10}
Let $\bbX=\{p_1,\dots,p_9\}\subseteq \bbP^3_\bbQ$
be the set of nine points given by
$p_1=(1:0:0:0)$,
$p_2=(1:1:0:0)$,
$p_3=(1:1:1:0)$,
$p_4= (1:1:-1:1)$,
$p_5=(1:-1:1:1)$,
$p_6=(1:-2:1:0)$,
$p_7=(1:-2:2:0)$,
$p_8=(1:-1:2:1)$,
and $p_9=(1:0:2:0)$.
We have $\HF_{\bbX}:\ 1\ 4\ 9\ 9\cdots$ and $r_\bbX=2$.
In this case the Hilbert functions of the K\"{a}hler
and Dedekind differents are given by
$$
\begin{array}{cc}
  \HF_{\vartheta_\bbX} &:\ 0\ 0\ 0\ 0\ 0\ 9\ 9\ \cdots \\
  \HF_{\delta_\bbX} &:\ 0\ 0\ 0\ 0\ 9\ 9\ 9\ \cdots
\end{array}
$$
It follows that
$\delta_\bbX =\bigoplus_{i\ge 2r_\bbX} R_i$,
and so $\bbX$ has minimal Dedekind different.
However, we have
$\vartheta_\bbX = \bigoplus_{i\ge 2r_\bbX+1} R_i
\subsetneq \delta_\bbX$.
\end{exam}

\bigskip\bigbreak
\section{The Cayley-Bacharach Property}

In this section we relate the algebraic structure
of the Dedekind different to the Cayley-Bacharach property
of a 0-dimensional scheme $\bbX$ in~$\bbP^n_K$.
First we use the notion of the maximal degree of a minimal
separator introduced in Section~1 to define the degree of
a point in $\bbX$.

\begin{defn}
For every $p_j\in\Supp(\bbX)$, the {\bf degree of $p_j$ in $\bbX$}
is defined as
$$
\deg_{\bbX}(p_j) := \min\big\{\, \mu_{\bbY/\bbX} \; \big| \;
\bbY \ \textrm{is a maximal $p_j$-subscheme of $\bbX$} \,\big\}.
$$
\end{defn}

Obviously, we have $\deg_{\bbX}(p_j) \le r_{\bbX}$ for all
$j = 1, \dots, s$. In case all points of $\Supp(\bbX)$ have
degree greater than some natural number $d$, we have the
following notion.

\begin{defn}
Let $d \ge 0$, let $\bbX \subseteq \bbP^n_K$ be a $0$-dimensional
scheme, and let $\Supp(\bbX)=\{p_1, \dots,p_s\}$.
We say that $\bbX$ has the {\bf Cayley-Bacharach property
of degree $d$} (in short, $\bbX$ has CBP($d$))
if every point $p_j\in\Supp(\bbX)$ has degree
$\deg_{\bbX}(p_j) \ge d+1$.
In the case that $\bbX$ has CBP($r_\bbX-1$) we also say that
$\bbX$ is a {\bf Cayley-Bacharach scheme}.
\end{defn}

If $\bbX$ has CBP($d$), then $\bbX$ has CBP($d-1$), and
every $0$-dimensional scheme~$\bbX$ with $\deg(\bbX) \ge 2$
has CBP($0$). Moreover, the number $r_{\bbX}-1$ is the
largest degree $d \ge 0$ such that $\bbX$ can have CBP($d$).
So, it suffices to consider the Cayley-Bacharach property
in degree $d \in \{0, \dots, r_{\bbX}-1\}$.

The following proposition gives a characterization of
Cayley-Bacharach property using standard sets of separators
of~$\bbX$.

\begin{prop}\label{propSec3.3}
Let $\bbX \subseteq \bbP^n_K$ be a $0$-dimensional scheme,
let $0 \le d \le r_{\bbX}-1$,
let $\Supp(\bbX) = \{p_1, \dots, p_s\}$, and let
$\varkappa_j = \dim \kappa(p_j)$.
Then the following statements are equivalent.
\begin{enumerate}
  \item The scheme $\bbX$ has CBP($d$).

  \item If $\bbY \subseteq \bbX$ is a maximal $p_j$-subscheme
  and $\{f_{j1},\dots,f_{j\varkappa_j}\}$  is a standard set
  of separators of~$\bbY$ in~$\bbX$, then there exists
  $k_j\in\{1 \dots, \varkappa_j\}$ such that
  $x_0^{r_{\bbX}-d} \nmid f_{jk_j}$.

  \item For all $p_j \in \Supp(\bbX)$, every maximal
  $p_j$-subscheme $\bbY \subseteq \bbX$ satisfies
   $$
   \dim_K(I_{\bbY/\bbX})_{d} < \varkappa_j.
   $$
\end{enumerate}
\end{prop}

\begin{proof}
Let $\bbY$ be a maximal $p_j$-subscheme of $\bbX$
and $\{f_{j1},\dots,f_{j\varkappa_j}\}$ a standard set of
separators of~$\bbY$ in~$\bbX$.
If we write $f_{jk_j}=x_0^{r_{\bbX}-\deg(f^*_{jk_j})}f^*_{jk_j}$
with $f^*_{jk_j} \in R_{\deg(f^*_{jk_j})}\setminus
x_0R_{\deg(f^*_{jk_j})-1}$ for $k_j = 1, \dots, \varkappa_j$,
then the set $\{f^*_{j1}, \dots, f^*_{j\varkappa_j}\}$ is a set
of minimal separators of~$\bbY$ in~$\bbX$.
Hence the equivalence of (a) and (b) follows.

Now we prove the equivalence of (a) and (c).
We always have $\dim_K(I_{\bbY/\bbX})_{i} \le \varkappa_j$
for $i \ge 0$. Moreover, we see that
$\dim_K(I_{\bbY/\bbX})_{d} = \varkappa_j$
if and only if $\deg(f^*_{jk_j}) \le d$ for all
$k_j=1, \dots,\varkappa_j$. This is equivalent to
$\deg_\bbX(p_j)\le d$. Thus the claim follows.
\end{proof}

Let us apply the proposition to a concrete case.

\begin{exam}\label{examSec3.4}
Let $\bbX \subseteq \bbP^2_\bbQ$ be the 0-dimensional scheme
of degree $8$ with support $\Supp(\bbX)=\{p_1,\dots,p_6\}$,
where $p_1=(1:0:0)$, $p_2=(1:1:0)$, $p_3=(1:0:1)$,
$p_4=(1:1:1)$, $p_5$ corresponds to
$\fP_5 =\langle X_1^2+3X_0^2, X_2\rangle$,
and $p_6$ corresponds to
$\fP_6 =\langle X_1-2X_0, 2X_0^2+X_2^2\rangle$.
We have $\varkappa_1=\cdots = \varkappa_4=1$ and
$\varkappa_5=\varkappa_6=2$.
The Hilbert functions of $\bbX$ and its subschemes are
$$
\begin{array}{ll}
  \HF_\bbX: & 1\ 3\ 6\ 8\ 8\cdots \\
  \HF_{\bbX\setminus\{p_j\}}: & 1\ 3\ 6\ 7\ 7\cdots (j=1,\dots,4)\\
  \HF_{\bbX\setminus\{p_5\}}: & 1\ 3\ 6\ 6\cdots \\
  \HF_{\bbX\setminus\{p_6\}}: & 1\ 3\ 5\ 6\ 6 \cdots.
\end{array}
$$
We see that
$\dim_K (I_{\bbX\setminus\{p_j\}/\bbX})_{r_{\bbX}-1}
=\dim_K (I_{\bbX\setminus\{p_j\}/\bbX})_{2} = 0
< \varkappa_j$ for $j=1,\dots,5$ and
$\dim_K (I_{\bbX\setminus\{p_6\}/\bbX})_{r_{\bbX}-1}
= 1< 2= \varkappa_6$.
Consequently, the scheme $\bbX$ is a Cayley-Bacharach
scheme by~Proposition~\ref{propSec3.3}.

Next we consider the subscheme $\bbY = \bbX\setminus\{p_4\}$ of~$\bbX$.
We have $\HF_{\bbY}: \ 1\ 3\ 6\ 7\ 7\cdots$ and $r_\bbY=3$.
The Hilbert functions of subschemes of~$\bbY$ are given by
$$
\begin{array}{ll}
  \HF_{\bbY\setminus\{p_j\}}: & 1\ 3\ 6\ 6\cdots (j=1,2)\\
  \HF_{\bbY\setminus\{p_3\}}: & 1\ 3\ 5\ 6\ 6\cdots \\
  \HF_{\bbY\setminus\{p_5\}}: & 1\ 3\ 5\ 5\cdots \\
  \HF_{\bbY\setminus\{p_6\}}: & 1\ 3\ 4\ 5\cdots.
\end{array}
$$
It follows that $\bbY$ has CBP($d$) for $d=0,1$.
But $\dim_K (I_{\bbY\setminus\{p_3\}/\bbY})_{r_{\bbY}-1}=1=\nu_3$
and $\dim_K (I_{\bbY\setminus\{p_6\}/\bbY})_{r_{\bbY}-1}=2=\nu_6$.
Therefore Proposition~\ref{propSec3.3} yields that
the scheme $\bbY$ is not a Cayley-Bacharach scheme.
\end{exam}

At this point we are ready to characterize Cayley-Bacharach
schemes in terms of their Dedekind differents.

\begin{thm}\label{thmSec3.5}
Let $\bbX\subseteq \bbP^n_K$ be a $0$-dimensional locally
Gorenstein scheme and let $\sigma$ be a homogeneous trace map
of degree zero of $Q^h(R)/L_0$.
Then $\bbX$ is a Cayley-Bacharach scheme if and only if,
for all $p_j\in\Supp(\bbX)$, every maximal $p_j$-subscheme
$\bbY \subseteq \bbX$ satisfies
$$
x_0^{r_{\bbX}-1}(I_{\bbY/\bbX})_{r_{\bbX}}
\nsubseteq (\delta^\sigma_\bbX)_{2r_{\bbX}-1}.
$$
\end{thm}

\begin{proof}
Suppose that $\bbX$ is a Cayley-Bacharach scheme.
By \cite[Proposition~3.2]{KL}, for every $j\in\{1,\dots,s\}$,
we find an element
$g^*_j \in (\fC^{\sigma}_{\bbX})_{-r_{\bbX}}$ such that
$g^*_j = x_0^{-2r_{\bbX}}\widetilde{g}^*_j$ with
$\widetilde{g}^*_j \in R_{r_{\bbX}} \setminus \{0\}$ and
$(\widetilde{g}^*_j)_{p_j} \in \calO_{\bbX,p_j}
\setminus \fm_{\bbX,p_j}$.
We assume for a contradiction that there is a maximal
$p_j$-subscheme $\bbY_j\subseteq\bbX$ such that
$$
x_0^{r_{\bbX}-1}(I_{\bbY_j/\bbX})_{r_{\bbX}}
\subseteq (\delta^\sigma_\bbX)_{2r_{\bbX}-1}.
$$
For such $j$, let $s_j$ be the socle element
in~$\mathcal{O}_{\bbX,p_j}$ corresponding to the scheme $\bbY_j$,
let $\{e_{j1},\dots, e_{j\varkappa_j}\}\subseteq
\mathcal{O}_{\bbX,p_j}$ be elements whose residue classes
form a $K$-basis of $\kappa(p_j)$, and let
$\{f_{j1},\dots, f_{j\varkappa_j}\}$
be the standard set of separators of $\bbY_j$ in~$\bbX$
w.r.t.~$s_j$ and $\{e_{j1},\dots, e_{j\varkappa_j}\}$.
We want to show that $x_0 \mid f_{jk}$ for
$k=1,\dots,\varkappa_j$. It suffices to show $x_0 \mid f_{j1}$,
since the other cases follow similarly.
We write
$$
\tilde{\imath}(f_{j1}) =
(0, \dots,0,e_{j1}s_jT_j^{r_{\bbX}},0, \dots,0)
$$
and put
$$
f := \tilde{\imath}^{-1}((0, \dots, 0,
e_{j1}(\widetilde{g}^*_j)^{-1}_{p_j}
s_jT_j^{r_{\bbX}}, 0, \dots, 0)).
$$
Then $0 \ne x_0^{r_{\bbX}-1}f \in
x_0^{r_{\bbX}-1}(I_{\bbY/\bbX})_{r_{\bbX}}$
and $f\widetilde{g}^*_j=x_0^{r_{\bbX}}f_{j1}$,
especially, $x_0^{r_{\bbX}-1}f \in
(\delta^\sigma_\bbX)_{2r_{\bbX}-1}$.
Also, we observe that
$$
x_0^{r_{\bbX}-1}f\cdot g^*_j
= x_0^{r_{\bbX}-1}f\cdot(x_0^{-2r_{\bbX}}\widetilde{g}^*_j)
=  x_0^{-r_{\bbX}-1}f\widetilde{g}^*_j
= x_0^{-r_{\bbX}-1+r_{\bbX}}f_{j1}
=  x_0^{-1}f_{j1}.
$$
So, it follows from the inclusion
$\fC^{\sigma}_{\bbX}\cdot\delta^{\sigma}_{\bbX} \subseteq R$
that $x_0^{-1}f_{j1} \in R_{r_{\bbX}-1}\setminus\{0\}$.
This implies $f_{j1} \in x_0R_{r_{\bbX}-1}$ or $x_0 \mid f_{j1}$.
Therefore Proposition~\ref{propSec3.3} yields that
$\bbX$ is not a Cayley-Bacharach scheme, a contradiction.

Conversely, suppose that $\bbX$ is not
a Cayley-Bacharach scheme.
Then there is a maximal $p_j$-subscheme
$\bbY_j\subseteq\bbX$ such that $\deg(f^*_{jk_j})\leq r_{\bbX}-1$
for all $k_j=1, \dots, \varkappa_j$. Notice that
$f_{jk_j} = x_0^{r_{\bbX}-\deg(f^*_{jk_j})}f^*_{jk_j}$ in
$x_0^{r_{\bbX}-\deg(f^*_{jk_j})}R_{\deg(f^*_{jk_j})}$
for all $k_j=1, \dots, \varkappa_j$.
As in Remark~\ref{remSec2.3}, we may write
$\fC^{\sigma}_{\bbX} =
\langle g_1,\dots,g_{\deg(\bbX)} \rangle_{K[x_0]}$,
where $g_k = x_0^{-2r_{\bbX}}\widetilde{g}_k$ with
$\widetilde{g}_k \in R_{2r_{\bbX}-n_k}$ for
$k = 1, \dots, \deg(\bbX)$ and $n_k\le r_\bbX$.
By~Lemma~\ref{lemSec1.8}, there are
$c_{j1}, \dots, c_{j\varkappa_j} \in K$ such that
$f_{j1}\cdot\widetilde{g}_k=\sum_{k_j=1}^{\varkappa_j}
c_{jk_j}x_0^{2r_{\bbX}-n_k}f_{jk_j}$.
We calculate
$$
\begin{aligned}
x_0^{r_{\bbX}-1}f_{j1}\cdot g_k
&= x_0^{r_{\bbX}-1}f_{j1}\cdot(x_0^{-2r_{\bbX}}\widetilde{g}_k)
=  x_0^{-r_{\bbX}-1}f_{j1}\widetilde{g}_k \\
&= x_0^{r_{\bbX}-n_k-1}
{\textstyle\sum\limits_{k_j=1}^{\varkappa_j}}c_{jk_j}f_{jk_j}\\
&= x_0^{r_{\bbX}-n_k}{\textstyle\sum\limits_{k_j=1}^{\varkappa_j}}
   c_{jk_j}x_0^{r_{\bbX} - \deg(f^*_{jk_j})-1}f^*_{jk_j}
   \in R_{2r_{\bbX}-n_k-1}.
\end{aligned}
$$
This implies $x_0^{r_{\bbX}-1}f_{j1} g_k \in R_{2r_{\bbX}-n_k-1}$
for every $k \in \{1, \dots, \deg(\bbX)\}$. Hence the element
$x_0^{r_{\bbX}-1}f_{j1}$ is contained in
$(\delta^{\sigma}_{\bbX})_{2r_{\bbX}-1}$.
Similarly, we can show that $x_0^{r_{\bbX}-1}f_{jk_j}$
is a homogeneous element of degree $2r_{\bbX}-1$
of~$\delta^{\sigma}_{\bbX}$ for all $k_j=2, \dots, \varkappa_j$.
Therefore we obtain
$$
x_0^{r_{\bbX}-1}(I_{\bbY_j/\bbX})_{r_{\bbX}}
= \langle x_0^{r_{\bbX}-1}f_{j1}, \dots,
x_0^{r_{\bbX}-1}f_{j\varkappa_j} \rangle_K
\subseteq (\delta^{\sigma}_{\bbX})_{2r_{\bbX}-1},
$$
in contradiction to the assumption that
$x_0^{r_{\bbX}-1}(I_{\bbY_j/\bbX})_{r_{\bbX}}\nsubseteq
(\delta^{\sigma}_{\bbX})_{2r_{\bbX}-1}$.
\end{proof}

The following corollary is an immediate consequence
of~Theorem~\ref{thmSec3.5}.

\begin{cor}\label{corSec3.6}
Let $\bbX\subseteq \bbP^n_K$ be a $0$-dimensional locally
Gorenstein scheme.
\begin{enumerate}
\item If $\bbX$ has $K$-rational support then
  it is a Cayley-Bacharach scheme if and only if
  for every subscheme $\bbY\subseteq\bbX$ of degree
  $\deg(\bbY) = \deg(\bbX)-1$ and for every separator
  $f_{\bbY}$ of~$\bbY$ in~$\bbX$ we have
  $x_0^{r_{\bbX}-1}f_{\bbY} \notin
  (\delta^{\sigma}_{\bbX})_{2r_{\bbX}-1}$.

\item If $\bbX$ has minimal Dedekind different
then it is a Cayley-Bacharach scheme.
\end{enumerate}
\end{cor}

Let us apply the corollary to some explicit cases.

\begin{exam} \label{examSec3.7}
Let $\bbX=\{p_1, \dots, p_6\}\subseteq \bbP^2_{\bbQ}$ be
the set of six points given in Example~\ref{examSec2.9}.
We know that $\bbX$ has minimal Dedekind different.
Therefore Corollary~\ref{corSec3.6}(b) yields that
$\bbX$ is a Cayley-Bacharach scheme.
Similarly, the set of nine points in $\bbP^3_\bbQ$ given in
Example~\ref{examSec2.9} is also a Cayley-Bacharach scheme.

Next we consider the 0-dimensional scheme
$\bbY \subseteq \bbP^2_\bbQ$ of degree $6$ with support $\Supp(\bbY)=\{p_1,\dots,p_5\}$,
where $p_1=(1:0:0)$, $p_2=(1:1:0)$, $p_3=(1:0:1)$,
$p_4=(1:1:1)$, and $p_5$ corresponds to
$\fP_5 =\langle X_1-2X_0, 2X_0^2+X_2^2\rangle$.
The Hilbert function of $\bbY$ is
$\HF_\bbY:\ 1\ 3\ 6\ 6\ \cdots$ and $r_\bbY=2$.
In this case the Hilbert function of the Dedekind different
is given by
$$
\HF_{\delta_\bbY^\sigma}:\ 0\ 0\ 0\ 0\ 6\ 6\ \cdots.
$$
It follows that $\bbY$ has minimal Dedekind different,
and so it is a Cayley-Bacharach scheme
by Corollary~\ref{corSec3.6}(b).
\end{exam}

For a Cayley-Bacharach scheme $\bbX\subseteq\bbP^n_K$,
the Hilbert function of the Dedekind different is described
in our next proposition.

\begin{prop}\label{propSec3.8}
Let $\bbX \subseteq \bbP^n_K$ be a 0-dimensional locally
Gorenstein Cayley-Bacharach scheme and let $\sigma$ be
a homogeneous trace map of degree zero of $Q^h(R)/\!L_0$.
Then the Hilbert function of~$\delta_\bbX^\sigma$
satisfies $\HF_{\delta_\bbX^\sigma}(i)=0$ for $i < r_{\bbX}$,
$\HF_{\delta_\bbX^\sigma}(i)=\deg(\bbX)$ for $i \ge 2r_{\bbX}$
and
$$
0 \le \HF_{\delta_\bbX^\sigma}(r_{\bbX})\le\cdots
\le \HF_{\delta_\bbX^\sigma}(2r_{\bbX}-1)
< \HF_{\delta_\bbX^\sigma}(2r_{\bbX})=\deg(\bbX).
$$
In this case, the regularity index of~$\delta_\bbX^\sigma$
is exactly~$2r_{\bbX}$.
\end{prop}

\begin{proof}
Since the scheme $\bbX$ is a Cayley-Bacharach scheme,
there are homogeneous elements $g_1^*, \dots, g_s^*$
in~$(\fC_{\bbX}^\sigma)_{-r_{\bbX}}$
such that $g^*_j = x_0^{-2r_{\bbX}}\widetilde{g}^*_j$
with $\widetilde{g}^*_j \in R_{r_{\bbX}}$ and
$(\widetilde{g}^*_j)_{p_j} \in
\calO_{\bbX,p_j}\setminus\fm_{\bbX,p_j}$
by \cite[Proposition~3.2]{KL}.
Let $h \in (\delta_\bbX^\sigma)_i$ with $i < r_{\bbX}$.
Then we have
$$
h \cdot g^*_j = x_0^{-2r_{\bbX}} h \widetilde{g}^*_j
 \in R_{i-r_{\bbX}} = \langle 0 \rangle
$$
for $j=1, \dots, s$. This implies $h \widetilde{g}^*_j = 0$,
in particular, $h_{p_j}\cdot (\widetilde{g}^*_j)_{p_j} = 0$
in~$\calO_{\bbX,p_j}$ for all $j \in \{1,\dots,s\}$.
Since $(\widetilde{g}^*_j)_{p_j}$ is a unit
of~$\calO_{\bbX,p_j}$ for $j = 1, \dots, s$,
we have to get $h_{p_j}=0$ for all $j = 1, \dots, s$.
In other words, we have $\tilde{\imath}(h) = 0$,
and so $h=0$ (as $\tilde{\imath}$ is an injection).
Subsequently, we get $\HF_{\delta_\bbX^\sigma}(i)=0$ for
$i < r_{\bbX}$.

Now, according to Proposition~\ref{propSec2.5},
we only need to show that
$\HF_{\delta_\bbX^\sigma}(2r_{\bbX}-1)<\deg(\bbX)$, i.e.,
$(\delta_\bbX^\sigma)_{2r_{\bbX}-1}\subsetneq R_{2r_{\bbX}-1}$.
But this follows from Theorem~\ref{thmSec3.5}, since otherwise
we would have
$x_0^{r_{\bbX}-1}(I_{\bbY/\bbX})_{r_{\bbX}} \subseteq
(\delta_\bbX^\sigma)_{2r_{\bbX}-1}$
for every maximal $p_j$-subscheme $\bbY \subseteq \bbX$,
and thus $\bbX$ would not be a Cayley-Bacharach scheme.
\end{proof}

\begin{rem}\label{remSec3.9}
The upper bound for the regularity index of the Dedekind
different given in Proposition~\ref{propSec2.5} is attained
for 0-dimensional locally Gorenstein Cayley-Bacharach schemes.
Moreover, a 0-dimensional locally Gorenstein Cayley-Bacharach
scheme $\bbX$ satisfies $\HF_{\delta_\bbX^\sigma}(r_{\bbX})>0$
if and only if $\bbX$ is arithmetically Gorenstein
(see \cite[Proposition~4.8]{KL}).
\end{rem}

\begin{prop} \label{propSec3.10}
Let $\bbX \subseteq \bbP^n_K$ be a $0$-dimensional locally
Gorenstein scheme, let  $0 \le d \le r_{\bbX}-1$, and let
$\sigma$ be a homogeneous trace map of degree zero
of~$Q^h(R)/\!L_0$.
If for every $p_j \in \Supp(\bbX)$ the maximal
$p_j$-subscheme $\bbY_j \subseteq \bbX$ satisfies
$$
x_0^{d}(I_{\bbY_j/\bbX})_{r_{\bbX}}
\nsubseteq (\delta_\bbX^\sigma)_{r_{\bbX}+d}
$$
then $\bbX$ has CBP($d$).
In particular, if $\HF_{\delta_\bbX^\sigma}(r_\bbX+d) =0$
then $\bbX$ has CBP($d$).
\end{prop}

\begin{proof}
Suppose for contradiction that $\bbX$ does not have CBP($d$).
There are a maximal $p_j$-subscheme $\bbY_j\subseteq \bbX$
and a set of minimal separators
$\{f^*_{j1}, \dots, f^*_{j\varkappa_j}\}$
of~$\bbY_j$ in~$\bbX$ such that $\deg(f^*_{jk_j})\le d$
for $k_j = 1,\dots,\varkappa_j$.
Set $f_{jk_j} := x_0^{r_{\bbX}-\deg(f_{jk_j}^*)}f^*_{jk_j}$
for $k_j \in \{1 \dots, \varkappa_j\}$.
Then the set $\{f_{j1}, \dots, f_{j\varkappa_j}\}$ is a standard set
of separators of~$\bbY_j$ in~$\bbX$. We write
$\fC_\bbX^\sigma=\langle g_1,\dots,g_{\deg(\bbX)} \rangle_{K[x_0]}$,
where $g_k = x_0^{-r_{\bbX}-n_k}\widetilde{g}_k$ with
$\widetilde{g}_k \in R_{r_{\bbX}}$ and $n_k\le r_{\bbX}$ for
$k = 1, \dots, \deg(\bbX)$ (see Remark~\ref{remSec2.3}).
We have
$$
\begin{aligned}
(x_0^{d}f_{jk}) \cdot (x_0^{-r_{\bbX}-n_l}\widetilde{g}_l)
&= x_0^{d-r_{\bbX}-n_l}f_{jk}\widetilde{g}_l
 = x_0^{d-n_l} {\textstyle\sum\limits_{k_j = 1}^{\varkappa_j}}
 c_{jk_j}f_{jk_j} \\
&= x_0^{d-n_l} {\textstyle\sum\limits_{k_j=1}^{\varkappa_j}}
   c_{jk_j} x_0^{r_{\bbX}-\deg(f_{jk_j}^*)} f_{jk_j}^* \\
&= x_0^{r_{\bbX}-n_l}
{\textstyle\sum\limits_{k_j=1}^{\varkappa_j}}
   c_{jk_j} x_0^{d-\deg(f_{jk_j}^*)}f_{jk_j}^*
\end{aligned}
$$
for some $c_{j1}, \dots, c_{j\varkappa_j} \in K$.
Since $r_{\bbX} - n_l \ge 0$ and $d - \deg(f_{jk_j}^*) \ge 0$,
this implies that $(x_0^{d}f_{jk}) \cdot
(x_0^{-r_{\bbX}-n_l}\widetilde{g}_l)\in R_{r_{\bbX}+d-n_l}$
for all $l = 1, \dots, \deg(\bbX)$.
Consequently, the element $x_0^{d}f_{jk}$ is contained
in~$(\delta_\bbX^\sigma)_{r_{\bbX}+d}$ for all
$k = 1, \dots, \varkappa_j$.
Therefore we get the inclusion
$x_0^{d}(I_{\bbY/\bbX})_{r_{\bbX}}
\subseteq (\delta_\bbX^\sigma)_{r_{\bbX}+d}$,
in~contradiction to our assumption.
\end{proof}

The following example shows that the converse
of~Proposition~\ref{propSec3.10} is not true in the
general case (except for the case $d = r_{\bbX}-1$).

\begin{exam}\label{examSec3.11}
Let $\bbX \subseteq \bbP^2_{\bbQ}$ be the set consisting
of the points $p_1=(1:0:0)$, $p_2=(1:1:0)$, $p_3=(1:2:0)$,
$p_4=(1:3:1)$, $p_5=(1:4:0)$, $p_6=(1:5:0)$, $p_7=(1:6:1)$,
and $p_8=(1:1:1)$.
It is easy to see that $\HF_{\bbX}: 1\ 3\ 5\ 7\ 8\ 8\ \cdots$
and $r_{\bbX}=4$. The Dedekind different is computed by
$$
\begin{aligned}
\delta_\bbX = \langle &\,
x_{1}^{2}x_{2}^{2} - \tfrac{20}{3} x_{1}x_{2}^{3} + 9x_{2}^{4},
\, x_{0}x_{1}^{6} - \tfrac{857}{3675} x_{1}^{7}, \\
&\, x_{0}^{5} - \tfrac{393}{100} x_{0}^{3}x_{1}^{2}
+ \tfrac{1431}{400} x_{0}^{2}x_{1}^{3}
- \tfrac{209}{200} x_{0}x_{1}^{4}
+ \tfrac{39}{400} x_{1}^{5}  - \tfrac{3919}{760} x_{2}^{5}
\, \rangle
\end{aligned}
$$
and its Hilbert function is
$\HF_{\delta_\bbX}:\ 0\ 0\ 0\ 0\ 1\ 3\ 5\ 7\ 8\ 8\ \cdots$.
Clearly, $\bbX$ is not arithmetically Gorenstein and
$\HF_{\delta_\bbX}(r_{\bbX}) \ne 0$.
Hence $\bbX$ is not a Cayley-Bacharach scheme
by~Remark~\ref{remSec3.9}.
Also, we can check that $\bbX$ has CBP($d$) for $0 \le d \le 2$.
Now the subscheme $\bbY_4 := \bbX \setminus \{p_4\}$ has
a separator of the form
$f_4 = x_0x_{1}^{2}x_{2} - 7x_0x_{1}x_{2}^{2} + 6x_0x_{2}^{3}$.
It is not difficult to verify that
$x_0^{r_{\bbX}-2}f_4 \in (\delta_\bbX)_{2r_{\bbX}-2}$.
Thus $\bbX$ has CBP($2$), but $x_0^{2}(I_{\bbY_4/\bbX})_{r_{\bbX}}
\subseteq (\delta_\bbX)_{r_{\bbX}+2}$.
\end{exam}

\medskip\bigbreak
\section{Dedekind's Formula}

In previous sections, we mainly considered the Dedekind different
to study the Cayley-Bacharach property of 0-dimensional locally
Gorenstein schemes. This different is a subideal of the conductor
of $R$ in the ring~$\prod_{j=1}^s\calO_{\bbX,p_j}[T_j]$.
In \cite{GKR}, Geramita \emph{et al.} characterized
a finite set of points to be a Cayley-Bacharach scheme
in terms of the conductor and showed that Dedekind's formula
for the conductor and the complementary module
always holds for finite sets of points.
In this section we generalize these results substantially.
We work over an arbitrary base field $K$, and let $\bbX$
be an arbitrary 0-dimensional subscheme of~$\bbP^n_K$.
Let the support of $\bbX$ be given by
$\Supp(\bbX)=\{p_1,\dots, p_s\}$.

\begin{defn}
Let $\widetilde{R} = \prod_{j=1}^s\calO_{\bbX,p_j}[T_j]$, and let
$\mathfrak{F}_{\widetilde{R}/R}$ be the ideal defined as
$$
\mathfrak{F}_{\widetilde{R}/R}
= \{\, f \in \widetilde{R} \, \mid\,
f\widetilde{R} \subseteq R \,\}.
$$
The ideal $\mathfrak{F}_{\widetilde{R}/R}$ is called the
{\bf conductor} of $R$ in $\widetilde{R}$.
\end{defn}

When the scheme $\bbX$ is reduced, the ring $\widetilde{R}$
is the integral closure of~$R$ in its full quotient ring,
and hence $\mathfrak{F}_{\widetilde{R}/R}$ is the conductor
of~$R$ in its integral closure in the traditional sense.
Furthermore, $\mathfrak{F}_{\widetilde{R}/R}$ is an ideal
of both~$R$ and~$\widetilde{R}$.
We recall from \cite[Proposition 2.9]{Kr2} the following
description of the conductor of~$R$ in~$\widetilde{R}$.

\begin{prop}\label{propSec4.2}
For $j \in\{1, \dots, s\}$ and $a\in\calO_{\bbX,p_j}$, let
$\mu(a) = \min\{i\in\bbN \mid (0, \dots, 0, aT_j^{i}, 0, \dots, 0)
\in \tilde{\imath}(R) \}$, where $\tilde{\imath}$
is the injection from $R$ to $Q^h(R)$, and let
$\nu(a) = \max\{\, \mu(ab) \mid  b \in \calO_{\bbX,p_j}
\setminus\{0\} \,\}$. Then, as an ideal of $R$, we have
$$
\mathfrak{F}_{\widetilde{R}/R} =
\big\langle \,  f_a \, \mid \, 1 \le j \le s, \,
a\in \calO_{\bbX,p_j}\setminus\{0\}  \, \big\rangle
$$
where $f_a$ is the preimage of
$(0, \dots, 0, aT_j^{\nu(a)}, 0, \dots, 0)$
under the injection~$\tilde{\imath}$.
\end{prop}

Some relations between the Dedekind different and the conductor
are given by the next proposition.

\begin{prop}\label{propSec4.3}
Let $\bbX \subseteq \bbP^n_K$ be a $0$-dimensional locally
Gorenstein scheme, and let $\sigma$ be a homogeneous trace map
of degree zero of $Q^h(R)/\!L_0$. Then we have
$$
\mathfrak{F}_{\widetilde{R}/R}^{2}
\,\subseteq\, \delta_\bbX^\sigma
\,\subseteq\, \mathfrak{F}_{\widetilde{R}/R}.
$$
\end{prop}

\begin{proof}
We know that $Q^h(R)=\prod_{j=1}^s\calO_{\bbX,p_j}[T_j,T_j^{-1}]$
and $(\fC_{\bbX}^\sigma)_i = Q^h(R)_i = \widetilde{R}_i$
for all $i \ge 0$.
This implies $\widetilde{R} \subseteq \fC_{\bbX}^\sigma$.
Thus we get
$$
\delta_\bbX^\sigma
= R:_{Q^h(R)} \fC_{\bbX}^\sigma
\subseteq R:_{Q^h(R)}\widetilde{R}
= R:_{\widetilde{R}}\widetilde{R}
= \mathfrak{F}_{\widetilde{R}/R}.
$$
Since $\bbX$ is a locally Gorenstein scheme, we have
$\Hom_{K}(\calO_{\bbX,p_j},K)\cong \calO_{\bbX,p_j}$
for all $j=1, \dots ,s$.
This implies the isomorphism $\widetilde{R}\cong
\underline{\Hom}_{K[x_0]}(\widetilde{R},K[x_0])$.
Hence we get
$\HF_{\underline{\Hom}_{K[x_0]}(\widetilde{R},K[x_0])}(i)
= \deg(\bbX)$ if and only if $i \ge 0$.
Let $f \in (\mathfrak{F}_{\widetilde{R}/R})_i$, let
$g \in (\fC_{\bbX}^\sigma)_k$, and let
$\varphi\in (\omega_R)_{k+1}$ such that
$g=\Phi(\varphi)$ where $\Phi$ was defined by~(\ref{mapSec2.1}).
Observe that
$(f\cdot\varphi)(\widetilde{R})= \varphi(f\widetilde{R})
\subseteq \varphi(R)\subseteq K[x_0]$.
This yields $f\cdot\varphi \in
\underline{\Hom}_{K[x_0]}(\widetilde{R},K[x_0])_{i+k}$.
If $f\cdot\varphi \ne 0$, then $\deg(f\cdot\varphi)=i+k\ge 0$.
Thus we have
$f\cdot g = f\cdot\Phi(\varphi)=\Phi(f\cdot\varphi)
\in \widetilde{R}$, and hence we get the inclusion
$\mathfrak{F}_{\widetilde{R}/R}\cdot\fC_{\bbX}^\sigma
\subseteq\widetilde{R}$.
Now we see that
$\mathfrak{F}_{\widetilde{R}/R}^2 \cdot \fC_{\bbX}^\sigma
\subseteq \mathfrak{F}_{\widetilde{R}/R} \widetilde{R}
\subseteq R$.
This yields the inclusion
$\mathfrak{F}_{\widetilde{R}/R}^2\subseteq\delta_\bbX^\sigma$.
Altogether, the claim follows.
\end{proof}

The Cayley-Bacharach property of a 0-dimensional scheme can be
characterized in terms of the conductor of~$R$ in $\widetilde{R}$,
as the following theorem shows.

\begin{thm}\label{thmSec4.4}
Let $\bbX\subseteq\bbP^n_K$ be a $0$-dimensional scheme,
and let $0 \le d \le r_{\bbX}-1$. Then $\bbX$ has CBP($d$)
if and only if $\mathfrak{F}_{\widetilde{R}/R} \subseteq
\bigoplus_{i\ge d+1} R_i$.
In particular, $\bbX$ is a Cayley-Bacharach scheme if and only if
$\mathfrak{F}_{\widetilde{R}/R} = \bigoplus_{i \ge r_{\bbX}} R_i$.
\end{thm}

\begin{proof}
Suppose that $\bbX$ has CBP($d$), but
$\mathfrak{F}_{\widetilde{R}/R} \nsubseteq
\bigoplus_{i\ge d+1} R_i$.
It follows from Proposition~\ref{propSec4.2} that
there are a non-zero element $a\in \calO_{\bbX,p_j}$
and a homogeneous element
$f_a \in \mathfrak{F}_{\widetilde{R}/R}\setminus\{0\}$
such that $\tilde{\imath}(f_a)
=(0,\dots,0,aT_j^{\nu(a)},0,\dots,0)$ and $\nu(a)\le d$.
So, we can find an element $b \in \calO_{\bbX,p_j}$ with
$s_j := ab \in \fG(\calO_{\bbX,p_j})\setminus\{0\}$.
By Proposition~\ref{propSec1.2}, there is a maximal
$p_j$-subscheme $\bbY$ of~$\bbX$ associated to
the socle element $s_j$.
We want to prove that $\mu_{\bbY/\bbX} \le d$.
Let $\varkappa_j= \dim_K \kappa(p_j)$, let
$\{e_{j1}, \dots, e_{j\varkappa_j}\} \subseteq \calO_{\bbX,p_j}$
be such that their residue classes form a $K$-basis of~$\kappa(p_j)$,
and let $\{f^*_{j1}, \dots,f^*_{j\varkappa_j}\}$ be the set
of minimal separators of~$\bbY$ in~$\bbX$ w.r.t. $s_j$ and
$\{e_{j1}, \dots, e_{j\varkappa_j}\}$.
Notice that $\tilde{\imath}(f^*_{jk_j}) =
(0, \dots, 0, e_{jk_j}s_jT_j^{\mu(e_{jk_j}s_j)}, 0, \dots, 0)$
and $\deg(f^*_{jk_j}) = \mu(e_{jk_j}s_j)$ for
$k_j = 1, \dots, \varkappa_j$.
Clearly, we have
$$
\begin{aligned}
\nu(s_j) &=
\max\{\, \mu(a's_j) \mid  a' \in \calO_{\bbX,p_j}
\setminus\{0\} \,\} \\
&\ge \max\{\, \mu(e_{jk}s_j) \,\mid\,
k = 1,\dots, \varkappa_j \,\}.
\end{aligned}
$$
This implies $\mu_{\bbY/\bbX} \le \nu(s_j)$.
Moreover, we also see that
$$
\nu(s_j) = \nu(ab) = \max\{\, \mu(abc) \,\mid\,
c \in \calO_{\bbX,p_j}, abc \ne 0 \,\} \le \nu(a) \le d.
$$
This yields $\mu_{\bbY/\bbX} \le \nu(s_j) \le d$.
Thus we get $\deg_\bbX(p_j)\le d$, and hence
$\bbX$ does not have CBP($d$), a contradiction.

Conversely, suppose that
$\mathfrak{F}_{\widetilde{R}/R}\subseteq\bigoplus_{i\ge d+1}R_i$.
Let $\bbY \subseteq \bbX$ be a maximal $p_j$-subscheme, and let
$\{f^*_{j1}, \dots,f^*_{j\varkappa_j}\}$ be the set of minimal
separators of~$\bbY$ in~$\bbX$ w.r.t. $s_j$ and
$\{e_{j1}, \dots, e_{j\varkappa_j}\}$.
As above, we always have
$$
\nu(s_j)\ge\max\{\mu(e_{jk}s_j)\mid k=1,\dots,\varkappa_j\}.
$$
Also, it is easy to check that $\mu(a+b)\le \max\{\mu(a),\mu(b)\}$
for all $a,b\in \calO_{\bbX,p_j}\setminus\{0\}$.
Let $a\in \calO_{\bbX,p_j}$ be such that $as_j \ne 0$.
Then we have $a \notin \fm_{\bbX,p_j}$ and we may write
$a = c_{j1}e_{j1} + \cdots + c_{j\varkappa_j}e_{j\varkappa_j}
\ ({\rm mod}\ \fm_{\bbX,p_j})$ for
$c_{j1},\dots,c_{j\varkappa_j}\in K$,
not all equal to zero.
We deduce $as_j = c_{j1}e_{j1}s_j + \cdots +
c_{j\varkappa_j}e_{j\varkappa_j}s_j$. Hence we have
$$
\begin{aligned}
\mu(as_j) &= \mu(c_{j1}e_{j1}s_j + \cdots +
c_{j\varkappa_j}e_{j\varkappa_j}s_j) \\
&\le \max\{\, \mu(e_{jk}s_j) \,\mid\,
k = 1,\dots, \varkappa_j \,\}.
\end{aligned}
$$
This implies
$\nu(s_j) = \max\{\mu(e_{jk}s_j)\mid k=1,\dots,\varkappa_j\}$.
Without loss of generality, we may assume that
$\nu(s_j) = \deg(f^*_{j1}) = \mu(e_{j1}s_j)$.
Thus we have $\nu(s_j) = \nu(e_{j1}s_j)$ and
$f^*_{j1} \in  \mathfrak{F}_{\widetilde{R}/R}$.
Since $\mathfrak{F}_{\widetilde{R}/R} \subseteq
\bigoplus_{i \ge d+1} R_i$, it follows that
$\nu(s_j) = \deg(f^*_{j1}) \ge d+1$.
From this we conclude that $\deg_{\bbX}(p_j) \ge d+1$ for all
$j=1,\dots,s$. In other words, the scheme~$\bbX$ has~CBP($d$).

Moreover, if we identify $R$ with its image under
$\tilde{\imath}$, we have $R_i = \widetilde{R}_i$
for all $i \ge r_{\bbX}$. Thus the ideal
$\bigoplus_{i \ge r_{\bbX}} R_i$ is an ideal of both~$R$
and~$\widetilde{R}$, and it is contained in the
conductor~$\mathfrak{F}_{\widetilde{R}/R}$.
Hence the additional claim follows.
\end{proof}

The inclusion
$\mathfrak{F}_{\widetilde{R}/R}^2\subseteq\delta_\bbX^\sigma$
in Proposition~\ref{propSec4.3} can be an equality
in the following case. In this case the converse
of Corollary~\ref{corSec3.6}(b) holds true.

\begin{prop}\label{propSec4.5}
Let $\bbX \subseteq \bbP^n_K$ be a $0$-dimensional locally
Gorenstein scheme, and let $\sigma$ be a homogeneous trace map
of degree zero of $Q^h(R)/L_0$.
Then the scheme $\bbX$ has minimal Dedekind different
if and only if $\bbX$ is a Cayley-Bacharach scheme and
$\mathfrak{F}_{\widetilde{R}/R}^2 =\delta_\bbX^\sigma$.
\end{prop}

\begin{proof}
Suppose that the scheme $\bbX$ has minimal Dedekind different.
Then the Dedekind different satisfies
$\HF_{\delta_\bbX^\sigma}(2r_\bbX-1)=0$.
By Corollary~\ref{corSec3.6}(b), the scheme $\bbX$
is a Cayley-Bacharach scheme.
So, Theorem~\ref{thmSec4.4} yields that
$\mathfrak{F}_{\widetilde{R}/R} =
\bigoplus_{i\ge r_\bbX} R_i$. Hence we have
$$
\mathfrak{F}_{\widetilde{R}/R}^2
= {\textstyle \bigoplus\limits_{i\ge 2r_\bbX}} R_i
= \delta_\bbX^\sigma.
$$
Conversely, if $\bbX$ is a Cayley-Bacharach scheme and
$\mathfrak{F}_{\widetilde{R}/R}^2 =\delta_\bbX^\sigma$,
then Theorem~\ref{thmSec4.4} implies the equality
$\delta_\bbX^\sigma = \bigoplus_{i\ge 2r_\bbX} R_i$.
Thus $\bbX$ has minimal Dedekind different.
\end{proof}

\begin{exam}
Let $\bbX$ and $\bbY$ be the two 0-dimensional reduced schemes
given in Example~\ref{examSec3.7}.
Both $\bbX$ and $\bbY$ have minimal Dedekind different.
Thus the Dedekind different equals to the square of
the conductor for these 0-dimensional schemes by the
preceding proposition.
\end{exam}

Our next theorem presents a generalization of Dedekind's formula
for the conductor $\mathfrak{F}_{\widetilde{R}/R}$ and
the Dedekind complementary module~$\fC_\bbX^\sigma$.
We use the notation $\nu_j = \dim_K\calO_{\bbX,p_j}$ for
all $j = 1, \dots, s$.

\begin{thm}\label{thmSec4.7}
Let $\bbX\subseteq\bbP^n_K$ be a $0$-dimensional locally
Gorenstein scheme with support $\Supp(\bbX)=\{p_1,\dots,p_s\}$,
and let $\sigma$ be a homogeneous trace map
of degree zero of $Q^h(R)/L_0$.
Further, we let $I_j$ be the homogeneous vanishing ideal
of~$\bbX$ at~$p_j$, and we let $\bbY_j$ be the subscheme
of~$\bbX$ defined by $I_{\bbY_j} = \bigcap_{k\ne j} I_k$
for $j = 1, \dots, s$.
Then the formula
$$
\mathfrak{F}_{\widetilde{R}/R} \cdot \fC^\sigma_\bbX
= \widetilde{R}
$$
holds true if one of the following conditions is satisfied:
\begin{enumerate}
  \item The scheme $\bbX$ is a Cayley-Bacharach scheme.

  \item For all $j \in \{1, \dots, s\}$, the Hilbert function
  of~$\bbY_j$ is of the form
   $$
   \HF_{\bbY_j}(i)=
   \begin{cases}
   \HF_{\bbX}(i)  & \textrm{if}\ i < \alpha_{\bbY_j/\bbX},\\
   \HF_{\bbX}(i)-\nu_j &\textrm{if}\ i\ge\alpha_{\bbY_j/\bbX}.
   \end{cases}
   $$
\end{enumerate}
\end{thm}

\begin{proof}
As in the proof of Proposition~\ref{propSec4.3},
we have the inclusion
$\mathfrak{F}_{\widetilde{R}/R}\cdot\fC_\bbX^\sigma
\subseteq\widetilde{R}$.
Now we prove the reverse inclusion if~(a) or~(b)
is satisfied.

(a)\ For every $j \in \{1, \dots, s\}$, we let
$\{\, e_{j1}, \dots, e_{j\nu_j} \,\}$
be a $K$-basis of~$\calO_{\bbX,p_j}$ and set
$\epsilon_{jk_j}:=(0, \dots, 0, e_{jk_j}, 0, \dots, 0)
\in \widetilde{R}$, where $k_j \in \{1, \dots, \nu_j\}$.
Then the elements $\{\epsilon_{11}, \dots, \epsilon_{s\nu_s}\}$
form a $K[x_0]$-basis of~$\widetilde{R}$.
Thus it is enough to show that
$\epsilon_{11}, \dots, \epsilon_{s\nu_s}$ are contained in
$\mathfrak{F}_{\widetilde{R}/R}\cdot \fC_\bbX^\sigma$.
Since $\bbX$ is a Cayley-Bacharach scheme,
for $j = 1, \dots, s$ we find
$g^*_j \in (\fC_{R/K[x_0]})_{-r_{\bbX}}$
such that $g^*_j = x_0^{-2r_{\bbX}} \widetilde{g}^*_j$,
where $\widetilde{g}^*_j\in R_{r_{\bbX}}$ and
$(\widetilde{g}^*_j)_{p_j}$ is a unit of~$\calO_{\bbX,p_j}$
(cf.~\cite[Proposition~3.2]{KL}).
By identifying $R$ with its image in $Q^h(R)$
under $\tilde{\imath}$, the element
$h_{jk_j} := (0, \dots, 0,(\widetilde{g}^*_j)^{-1}_{p_j}
e_{jk_j}T_j^{r_{\bbX}}, 0, \dots, 0)$
is contained in~$R_{r_{\bbX}}\setminus\{0\}$
for all $j \in \{1, \dots, s\}$ and $k_j\in \{1,\dots,\nu_j\}$.
We see that
$$
h_{jk_j}\cdot g^*_j = x_0^{-2r_{\bbX}} h_{jk_j}\widetilde{g}^*_j
= x_0^{-2r_{\bbX}} (0, \dots, 0, e_{jk_j}T_j^{2r_{\bbX}},
0, \dots, 0)
= \epsilon_{jk_j} \in \widetilde{R}.
$$
By Theorem~\ref{thmSec4.4}, we have
$\mathfrak{F}_{\widetilde{R}/R}=\bigoplus_{i\geq r_{\bbX}}R_i$.
This implies $h_{11}, \dots, h_{s\nu_s}
\in \mathfrak{F}_{\widetilde{R}/R}$.
Therefore we obtain $\epsilon_{11}, \dots, \epsilon_{s\nu_s} \in
\mathfrak{F}_{\widetilde{R}/R} \cdot \fC_\bbX^\sigma$,
as desired.

(b)\ In a similar fashion, we proceed to show that
$\epsilon_{11}, \dots, \epsilon_{s\nu_s} \in
\mathfrak{F}_{\widetilde{R}/R}\cdot\fC_\bbX^\sigma$.
For $j=1, \dots, s$, let $\overline{\sigma}_j$ denote
the trace map of the algebra~$\calO_{\bbX,p_j}/K$
(associated to~$\sigma$),
and let $\{e'_{j1}, \dots, e'_{j\nu_j}\}$
be the $K$-basis of~$\calO_{\bbX,p_j}$ which is
dual to the $K$-basis $\{e_{j1}, \dots, e_{j\nu_j}\}$
w.r.t.~$\overline{\sigma}_j$.
W.l.o.g., we may assume that $e'_{j1}$ is a unit
of~$\calO_{\bbX,p_j}$ for all $j \in \{1, \dots, s\}$.
Note that the subscheme $\bbY_j$ has degree
$\deg(\bbY_j)=\deg(\bbX) - \nu_j$ for all $j = 1, \dots, s$.
It follows from the assumption that
$\alpha_{\bbY_j/\bbX} =\mu(e_{j1})=\cdots=\mu(e_{j\nu_j})$.
Then we have $I_{\bbY_j/\bbX} =
\langle\, f^*_{j1}, \dots, f^*_{j\nu_j} \,\rangle$,
where
$$
f^*_{jk_j} = \tilde{\imath}^{-1}((0, \dots, 0,
e_{jk_j}T_j^{\alpha_{\bbY_j/\bbX}}, 0, \dots, 0))
$$
for $k_j=1, \dots, \nu_j$.
The set $\{ (0, \dots, 0, a T_j^{\alpha_{\bbY_j/\bbX}}, 0, \dots, 0)
\,\mid\, a\in \calO_{\bbX,p_j} \}$ is
the image of~$(I_{\bbY_j/\bbX})_{\alpha_{\bbY_j/\bbX}}$
in~$\widetilde{R}$.
This implies
$\nu(a) = \mu(a) = \alpha_{\bbY_j/\bbX}$ for every non-zero
element $a\in \calO_{\bbX,p_j}$.
Thus~Proposition~\ref{propSec4.2} yields that
$I_{\bbY_j/\bbX} \subseteq \mathfrak{F}_{\widetilde{R}/R}$.

Obviously, we have $f^*_{jk_j} \notin \langle x_0 \rangle$
and its image $\overline{f}^*_{jk_j}$
in~$\overline{R}=R/\langle x_0\rangle$ is
a non-zero element for $k_j = 1, \dots, \nu_j$.
If there exist elements $a_{j1}, \dots, a_{j\nu_j}\in K$,
not all equal to zero, such that
$\sum_{k_j=1}^{\nu_j} a_{jk_j} \overline{f}^*_{jk_j} =0$,
then $f = \sum_{k_j=1}^{\nu_j} a_{jk_j} f^*_{jk_j}$ is contained
in~$(I_{\bbY_j/\bbX})_{\alpha_{\bbY_j/\bbX}}\setminus\{0\}$,
and we get $\overline{f} = 0$.
This means $f=x_0h \in x_0R_{\alpha_{\bbY_j/\bbX}-1}$
for some $h \in R_{\alpha_{\bbY_j/\bbX}-1}\setminus\{0\}$.
Since the ideal $I_{\bbY_j/\bbX}$ is saturated,
\cite[Lemma~1.2]{Kr1} implies
$h \in I_{\bbY_j/\bbX}\setminus\{0\}$, a contradiction.
Thus we have shown that the set
$\{\overline{f}^*_{j1}, \dots, \overline{f}^*_{j\nu_j}\}$
is $K$-linearly independent.

Consequently, there is a homogeneous $K$-linear map
$\overline{\varphi}_{j1}: \overline{R}\rightarrow K$
of degree $-\alpha_{\bbY_j/\bbX}$ with
$\overline{\varphi}_{j1}(\overline{f}^*_{j1}) \ne 0$
and $\overline{\varphi}_{j1}(\overline{f}^*_{jk_j}) = 0$
for $k_j = 2, \dots, \nu_j$.
Using the epimorphism $\omega_R(1) \twoheadrightarrow
\underline{\Hom}_K(\overline{R},K)$, we can lift
$\overline{\varphi}_{j1}$ to obtain a homogeneous element
$\varphi_{j1} \in (\omega_R)_{-\alpha_{\bbY_j/\bbX}+1}$
with $\varphi_{j1}(f^*_{j1}) \ne 0$ and
$\varphi_{j1}(f^*_{jk_j}) = 0$ for $k_j = 2, \dots, \nu_j$.
Clearly, the set $\{x_0^{r_{\bbX}-\mu(e_{11})}f^*_{11}, \dots,
x_0^{r_{\bbX}-\mu(e_{s\nu_s})}f^*_{s\nu_s} \}$ forms
a $K$-basis of the $K$-vector space~$R_{r_{\bbX}}$.
We write
$\varphi_{j1}(x_0^{r_{\bbX}-\mu(e_{j'k_{j'}})}f^*_{j'k_{j'}})
=c_{j'k_{j'}}x_0^{r_{\bbX}-\alpha_{\bbY_j/\bbX}}$
for all $j' = 1, \dots, s$ and $k_{j'} = 1, \dots, \nu_{j'}$.
By~Proposition~\ref{propSec3.4}, we have
$$
g_{j1} := \Phi(\varphi_{j1})
    = \big({\textstyle\sum\limits_{k_1=1}^{\nu_1}}
    c_{1k_1}e'_{1k_1}T_1^{-\alpha_{\bbY_j/\bbX}}, \dots,
    {\textstyle\sum\limits_{k_s=1}^{\nu_s}}
    c_{sk_s}e'_{sk_s}T_s^{-\alpha_{\bbY_j/\bbX}}
   \big) \in \fC_\bbX^\sigma.
$$
Since $e'_{j1}$ is a unit of~$\calO_{\bbX,p_j}$ and
$c_{j1} \in K\setminus\{0\}$, for $k_j = 1, \dots, \nu_j$ we set
$$
h_{jk_j} = \tilde{\imath}^{-1}((0, \dots, 0,
(e'_{j1}c_{j1})^{-1} e_{jk_j}
T_j^{\alpha_{\bbY_j/\bbX}}, 0, \dots, 0)).
$$
Then $h_{j1}, \dots, h_{j\nu_j} \in I_{\bbY_j/\bbX}
\subseteq \mathfrak{F}_{\widetilde{R}/R}$.
In $\widetilde{R}$, we have
$$
\begin{array}{ll}
  h_{jk_j} \cdot g_{j1}
  & = (0, \dots, 0, (e'_{j1}c_{j1})^{-1}e_{jk_j}
    {\textstyle\sum\limits_{l_j = 1}^{\nu_j}}
    c_{jl_j}{e'}_{jl_j}, 0, \dots, 0) \\
  & = (0, \dots, 0, e_{jk_j}, 0, \dots, 0)
  = \epsilon_{jk_j},
\end{array}
$$
since $c_{j2} = \cdots = c_{j\nu_j} = 0$. Thus we obtain
$\epsilon_{jk_j}\in \mathfrak{F}_{\widetilde{R}/R}
\cdot \fC_\bbX^\sigma$, as was to be shown.
\end{proof}

When we specialize to the case of sets of points,
the condition~(b) of~Theorem~\ref{thmSec4.7}
is satisfied. Therefore we recover the following result
of~A.V.~Geramita~\textit{et~al.}
(see~\cite[Proposition~3.15]{GKR}).

\begin{cor}
Let $\bbX=\{p_1, \dots, p_s\}\subseteq \bbP^n_K$ be a set
of $s$ distinct $K$-rational points.
Then we have $\mathfrak{F}_{\widetilde{R}/R} \cdot \fC_\bbX
= \widetilde{R}$.
\end{cor}

We end this section with some straightforward consequences
of the theorem.

\begin{cor}\label{corSec4.09}
Let $\bbX \subseteq \bbP^n_K$ be a $0$-dimensional locally
Gorenstein scheme, let $0\le d\le r_\bbX-1$, and let $\sigma$
be a homogeneous trace map of degree zero of $Q^h(R)/L_0$.
If $\bbX$ has CBP($d$), then $\HF_{\delta^\sigma_\bbX}(d)=0$.
\end{cor}

\begin{proof}
If $\HF_{\delta^\sigma_\bbX}(d)\ne 0$, then there exists a
non-zero homogeneous element $h$ in~$(\delta^\sigma_\bbX)_{d}$.
Proposition~\ref{propSec4.3} yields that
$h \in (\mathfrak{F}_{\widetilde{R}/R})_d$.
By~Theorem~\ref{thmSec4.4}, the scheme $\bbX$ does not
have CBP($d$), a contradiction.
\end{proof}

\begin{cor}\label{corSec4.10}
Let $\bbX=\{p_1, \dots,p_s\}\subseteq\bbP^n_K$ be a set
of $s$ distinct $K$-rational points,
and for $j = 1, \dots, s$ let $f_j$ be the separator
of~$\bbX\setminus\{p_j\}$ in~$\bbX$
such that $f_j(p_j)=1$ and $f_j(p_k)=0$ for $k\ne j$.
Then $\bbX$ is a Cayley-Bacharach scheme if and only if
$x_0^{r_{\bbX}-2}f_{j} \notin (\delta_\bbX)_{2r_{\bbX}-2}$
for all $j = 1, \dots, s$.
\end{cor}

\begin{proof}
It is clear that $x_0^{r_{\bbX}-1}f_{j} \in
(\delta_\bbX)_{2r_{\bbX}-1}$
if $x_0^{r_{\bbX}-2}f_{j} \in (\delta_\bbX)_{2r_{\bbX}-2}$.
By~Corollary~\ref{corSec3.6}(a), we get
$x_0^{r_{\bbX}-2}f_{j} \notin (\delta_\bbX)_{2r_{\bbX}-2}$
for every $j \in \{1, \dots, s\}$ if $\bbX$ is
a Cayley-Bacharach scheme.
Conversely, if $\bbX$ is not a Cayley-Bacharach scheme,
we find a minimal separator $f_j^* \in R$ such that
$d_j = \deg(f^*_j) \le r_{\bbX}-1$ and $f_j^*(p_j) = 1$.
Notice that $f_j^* \in \mathfrak{F}_{\widetilde{R}/R}$.
By Proposition~\ref{propSec4.3}, we get
$(f_{j}^*)^2 \in \mathfrak{F}_{\widetilde{R}/R}^2
\subseteq \delta_{\bbX}$.
Moreover, we have $x_0^{d_j}f_{j}^* = (f_{j}^*)^2$
and $f_{j} = x_0^{r_{\bbX}-d_j} f_{j}^* \in R_{r_{\bbX}}$.
This implies that
$x_0^{r_{\bbX}-2}f_{j}
= x_0^{2r_{\bbX}-2d_j-2}(x_0^{d_j}f_{j}^*)
\in (\delta_{\bbX})_{2r_{\bbX}-2}$.
Therefore the proof is complete.
\end{proof}

\medskip\bigbreak
\section{The Trace of the Dedekind Complementary Module}

In this section we let $\bbX$ be a 0-dimensional locally
Gorenstein scheme in~$\bbP^n_K$, let
$\Supp(\bbX)=\{p_1,\dots,p_s\}$, and let $\sigma$ be
a fixed homogeneous trace map of degree zero of
the graded algebra~$Q^h(R)/L_0$.

\begin{defn}
The {\bf trace} of the Dedekind complementary module
$\fC_\bbX^\sigma$, denoted ${\rm tr}(\fC_\bbX^\sigma)$,
is the sum of the ideals $\phi(\fC_\bbX^\sigma)$
with $\phi \in \underline{\Hom}_R(\fC_\bbX^\sigma,R)$, i.e.,
$$
{\rm tr}(\fC_\bbX^\sigma) =
\sum_{\phi \in \underline{\Hom}_R(\fC_\bbX^\sigma,R)}
\varphi(\fC_\bbX^\sigma).
$$
\end{defn}

The following remark collects some basic properties of
${\rm tr}(\fC_\bbX^\sigma)$. For the general theory
of traces of modules we refer to \cite{HHS,Lin16}.

\begin{rem}
Notice that we have $\omega_R(1) \cong \fC_\bbX^\sigma$,
and so ${\rm tr}(\fC_\bbX^\sigma) = {\rm tr}(\omega_R(1))$.
Moreover, there is an isomorphism of graded $R$-modules
$$
\delta_\bbX^\sigma
= R:_R \fC_\bbX^\sigma
\cong \underline{\Hom}_R(\fC_\bbX^\sigma,R)
$$
given by $h \mapsto \mu_h$,
where $\mu_h:\fC_\bbX^\sigma \rightarrow R$ is the multiplication
by~$h$. This implies that
$$
{\rm tr}(\fC_\bbX^\sigma)
= \delta_\bbX^\sigma\cdot \fC_\bbX^\sigma.
$$
In particular, the scheme $\bbX$ is arithmetically Gorenstein
if and only if ${\rm tr}(\fC_\bbX^\sigma) = R$.
\end{rem}

The relation between the trace
${\rm tr}(\fC_\bbX^\sigma)$ and the conductor of $R$
in the graded ring
$\widetilde{R}=\prod_{j=1}^s\calO_{\bbX,p_j}[T_j]$
is given by the following proposition.

\begin{prop} \label{propSec5.1}
Let $\mathfrak{F}_{\widetilde{R}/R}$ be the conductor of~$R$
in~$\widetilde{R}$.

\begin{enumerate}
  \item If $\bbX$ is a Cayley-Bacharach scheme, then
  $\mathfrak{F}_{\widetilde{R}/R}\subseteq
  {\rm tr}(\fC_\bbX^\sigma)$.

  \item The scheme $\bbX$ is a Cayley-Bacharach scheme
  such that $\mathfrak{F}_{\widetilde{R}/R}
  ={\rm tr}(\fC_\bbX^\sigma)$
  if and only if $\bbX$ has minimal Dedekind different.
\end{enumerate}
\end{prop}

\begin{proof}
Suppose that $\bbX$ is a Cayley-Bacharach scheme.
Then Theorem~\ref{thmSec4.4} yields
$\mathfrak{F}_{\widetilde{R}/R} = \bigoplus_{i\ge r_\bbX} R_i$.
Furthermore, by \cite[Proposition~3.2]{KL},
for every $j\in\{1,\dots,s\}$, we find an element
$x_0^{-2r_{\bbX}}\widetilde{g}^*_j \in
(\fC^{\sigma}_{\bbX})_{-r_{\bbX}}\setminus \{0\}$ such that
$\widetilde{g}^*_j \in R_{r_{\bbX}}$ and
$(\widetilde{g}^*_j)_{p_j}$ is a unit of~$\calO_{\bbX,p_j}$.
It is also clear that
$\bigoplus_{i\ge 2r_\bbX} R_i \subseteq \delta_\bbX^\sigma$.
Hence we have
$R_{r_\bbX} \subseteq \delta_\bbX^\sigma\cdot \fC_\bbX^\sigma
={\rm tr}(\fC_\bbX^\sigma)$, and claim (a) follows.

Now we prove (b).
Assume that $\bbX$ is a Cayley-Bacharach scheme such that
$\mathfrak{F}_{\widetilde{R}/R}={\rm tr}(\fC_\bbX^\sigma)$.
For a contradiction suppose that
$\bbX$ does not have minimal Dedekind different.
This implies $\HF_{\delta_\bbX^\sigma}(2r_\bbX-1)\ne 0$.
Let $h \in (\delta_\bbX^\sigma)_{2r_\bbX-1}\setminus\{0\}$.
Then there is an index $j\in\{1,\dots,s\}$ such that
$h_{p_j}\ne 0$ in~$\calO_{\bbX,p_j}$.
Let $\widetilde{g}^*_j\in R_{r_{\bbX}}$ be given as
in the proof of~(a). Then we have
$(h\widetilde{g}^*_j)_{p_j} \ne 0$ in~$\calO_{\bbX,p_j}$.
It follows that
$0\ne hx_0^{-2r_{\bbX}}\widetilde{g}^*_j
\in ({\rm tr}(\fC_\bbX^\sigma))_{r_\bbX-1}$.
But ${\rm tr}(\fC_\bbX^\sigma)=\mathfrak{F}_{\widetilde{R}/R}
= \bigoplus_{i\ge r_\bbX} R_i$, which is impossible.

Conversely, suppose that the scheme $\bbX$ has
minimal Dedekind different.
Then $\bbX$ is a Cayley-Bacharach scheme
by Corollary~\ref{corSec3.6}(b).
Moreover, the Dedekind different satisfies
$\delta_\bbX^\sigma = \bigoplus_{i\ge 2r_\bbX} R_i$.
It follows that $({\rm tr}(\fC_\bbX^\sigma))_{r_\bbX-1}=
(\delta_\bbX^\sigma)_{2r_\bbX-1}(\fC_\bbX^\sigma)_{-r_\bbX}=
\langle 0\rangle$.
Therefore the equality
$\mathfrak{F}_{\widetilde{R}/R} ={\rm tr}(\fC_\bbX^\sigma)$
follows from claim~(a).
\end{proof}

\begin{exam}
Let $\bbX=\{p_1,\dots,p_9\}\subseteq \bbP^3_\bbQ$
be the set of nine points given in Example~\ref{examSec2.10}.
We saw that $\HF_{\bbX}:\ 1\ 4\ 9\ 9\cdots$ and $r_\bbX=2$.
Moreover, $\bbX$ has minimal Dedekind different,
and so it is a Cayley-Bacharach scheme.
In addition, we have
${\rm tr}(\fC_\bbX) = \bigoplus_{i\ge 2} R_i$
by Proposition~\ref{propSec5.1}(b).
\end{exam}

In view of the theory of nearly and almost Gorenstein rings
given in the papers \cite{BF97,GTT,HHS}, we introduce
the following two special classes of 0-dimensional schemes
in~$\bbP^n_K$. Note that $\mathfrak{m}$ denotes the
homogeneous maximal ideal of~$R$.

\begin{defn}
Let $\bbX$ be a 0-dimensional locally Gorenstein scheme
in~$\bbP^n_K$.
\begin{enumerate}
  \item The scheme $\bbX$ is called a {\bf nearly Gorenstein
  scheme} if $\mathfrak{m}\subseteq {\rm tr}(\fC_\bbX^\sigma)$.

  \item The scheme $\bbX$ is called an {\bf almost Gorenstein
  scheme} if there is an exact sequence of graded $R$-modules
  $$
  0 \longrightarrow R \longrightarrow \fC_\bbX^\sigma(-r_\bbX)
  \longrightarrow C \longrightarrow 0
  $$
  with $\mathfrak{m}\cdot C = \langle 0\rangle$.
\end{enumerate}
\end{defn}

Note that every arithmetically Gorenstein scheme $\bbX$
is nearly Gorenstein and almost Gorenstein, and that
$\bbX$ is a Cayley-Bacharach scheme if it is an almost
Gorenstein scheme (since there exists an element
$g\in (\fC_\bbX^\sigma)_{-r_\bbX}$ with
${\rm Ann}_R(g)=\langle 0\rangle$).

In our setting, the class of almost Gorenstein schemes is smaller
than that of nearly Gorenstein schemes. The following proof of
this property mimics the proof of~\cite[Proposition~6.1]{HHS}
for local rings.

\begin{prop}\label{propSec5.06}
If $\bbX$ is an almost Gorenstein scheme, then
it is a nearly Gorenstein scheme and
$\HF_{\delta_\bbX^\sigma}(r_\bbX+1)=\HF_\bbX(1)$.
\end{prop}

\begin{proof}
If $\bbX$ is arithmetically Gorenstein, we have
$\HF_{\delta_\bbX^\sigma}(r_\bbX+1)=\HF_\bbX(1)$
by \cite[Proposition~5.8]{KL}.
So, we may assume that $\bbX$ is not
arithmetically Gorenstein.
Then $C\ne \langle 0\rangle$ and
$\mathfrak{m}\cdot C= \langle 0\rangle$, and so
$\underline{\Hom}_R(C,R)=\langle 0\rangle$.
By applying the functor $\underline{\Hom}_R(-,R)$
to the homogeneous exact sequence
$$
0 \longrightarrow R
\stackrel{\theta}{\longrightarrow}
\fC_\bbX^\sigma(-r_\bbX) \longrightarrow
C \longrightarrow 0
$$
we get the exact sequence
$$
0 \longrightarrow \delta_\bbX^\sigma(r_\bbX)
\stackrel{\theta^*}{\longrightarrow}  R
\longrightarrow {\rm \underline{Ext}}^1_R(C,R).
$$
Here the map $\theta^*:
\delta_\bbX^\sigma(r_\bbX) \rightarrow R$
is given by $h\mapsto h\theta(1)$ and
$\deg(\theta(1))=-r_\bbX$.  Also, we have
$\mathfrak{m}\cdot{\rm\underline{Ext}}^1_R(C,R)
= \langle0\rangle$.
This implies $\mathfrak{m} \subseteq
\delta_\bbX^\sigma\cdot \theta(1) \subseteq
\delta_\bbX^\sigma\cdot \fC_\bbX^\sigma
={\rm tr}(\fC_\bbX^\sigma)$, and so
$\bbX$ is a nearly Gorenstein scheme.
Moreover, Remark~\ref{remSec3.9} yields
$\HF_{\delta_\bbX^\sigma}(r_\bbX)=0$,
and so we have
$\mathfrak{m} = \delta_\bbX^\sigma\cdot\theta(1)$.
Consequently, we get
$\HF_{\delta_\bbX^\sigma}(r_\bbX+1)
= \HF_{\delta_\bbX^\sigma\cdot \theta(1)}(1)
= \HF_{\mathfrak{m}}(1) =\HF_\bbX(1)$,
since ${\rm Ann}_R(\theta(1))=\langle 0\rangle$.
\end{proof}

Notice that every nearly Gorenstein scheme $\bbX$
satisfies $\HF_{\delta_\bbX^\sigma}(r_\bbX+1)\ne 0$,
because otherwise we would have
$\mathfrak{m}_1 \nsubseteq
({\rm tr}(\fC_\bbX^\sigma))_1=\langle 0\rangle$.
Hence this implies the following corollary.

\begin{cor}
If $\bbX$ has minimal Dedekind different and $r_\bbX\ge 2$,
then it is not a nearly Gorenstein scheme.
\end{cor}

It is natural to ask: If $\bbX$ is a nearly Gorenstein scheme,
when is $\bbX$ an almost Gorenstein scheme?
In the case that
$\Delta_\bbX = \deg(\bbX)-\HF_\bbX(r_\bbX-1)=1$,
we have the following answer to this question.

\begin{prop}\label{propSec5.5}
Let $\bbX$ be a 0-dimensional locally Gorenstein scheme
in~$\bbP^n_K$ such that $\Delta_\bbX=1$.
Then the following conditions are equivalent.
\begin{enumerate}
  \item $\bbX$ is an almost Gorenstein scheme.
  \item $\bbX$ is a nearly Gorenstein Cayley-Bacharach scheme.
\end{enumerate}
\end{prop}

\begin{proof}
It suffices to prove the implication ``(b)$\Rightarrow$(a)''.
Suppose that $\bbX$ is a nearly Gorenstein Cayley-Bacharach scheme.
We may assume that $\bbX$ is not arithmetically Gorenstein.
By \cite[Proposition~5.8]{KL}, we have
$(\delta_\bbX^\sigma)_{r_\bbX} = \langle 0\rangle$.
Since $\bbX$ is nearly Gorenstein, we have
$\mathfrak{m} = {\rm tr}(\fC_\bbX^\sigma)$.
This implies
$$
(\delta_\bbX^\sigma)_{r_\bbX+1}\cdot (\fC_\bbX^\sigma)_{-r_\bbX}
= \mathfrak{m}_1.
$$
Since $\bbX$ is a Cayley-Bacharach scheme
and $\Delta_\bbX =1$, \cite[Proposition~4.12]{KL} shows that
there exists an element
$g\in (\fC_\bbX^\sigma)_{-r_\bbX}$ such that
$(\fC_\bbX^\sigma)_{-r_\bbX} = \langle g\rangle_K$
and ${\rm Ann}_R(g)=\langle 0\rangle$.
Hence we have
$(\delta_\bbX^\sigma)_{r_\bbX+1}\cdot g = \mathfrak{m}_1$.
Consider the exact sequence of graded $R$-modules
$$
0 \longrightarrow R
\stackrel{\theta}{\longrightarrow}
\fC_\bbX^\sigma(-r_\bbX) \longrightarrow
C \longrightarrow 0
$$
where $\theta: R\rightarrow \fC_\bbX^\sigma(-r_\bbX)$
is the injection given by $1\mapsto g$ and
$C = \fC_\bbX^\sigma(-r_\bbX)/ \langle g\rangle_R$.
Now we want to show that $\mathfrak{m}\cdot C = 0$.
Clearly, $\mathfrak{m}\cdot C = 0$ if and only if
$\mathfrak{m}\cdot\fC_\bbX^\sigma(-r_\bbX)
= \mathfrak{m}\cdot g$.
This is equivalent to
$\mathfrak{m}_1\cdot\fC_\bbX^\sigma(-r_\bbX)
\subseteq \mathfrak{m}\cdot g$.
Let $i\ge 0$, $g' \in (\fC_\bbX^\sigma(-r_\bbX))_i$,
and $\ell \in \mathfrak{m}_1$.
Set $\ell = h \cdot g$ with $h\in (\delta_\bbX^\sigma)_{r_\bbX+1}$,
since $(\delta_\bbX^\sigma)_{r_\bbX+1}\cdot g = \mathfrak{m}_1$.
We have
$$
\ell\cdot g' = h \cdot g \cdot g' = (h \cdot g') \cdot g.
$$
Since $h\cdot g' \in R_1$,
we get $\ell\cdot g' \in \mathfrak{m}\cdot g$.
It follows that $\mathfrak{m}_1\cdot\fC_\bbX^\sigma(-r_\bbX)
\subseteq \mathfrak{m}\cdot g$, and hence
$\mathfrak{m}\cdot C = 0$, as desired.
\end{proof}

Let us apply this proposition to an explicit example.

\begin{exam}\label{examS5.07}
Let $\bbX=\{p_1,\dots,p_7\}\subseteq \bbP^2_\bbQ$
be the set of seven points given by
$p_1= (1:0:0)$
$p_2=(1:1:0)$,
$p_3=(1:0:1)$,
$p_4=(1:1:1)$,
$p_5=(1:0:2)$,
$p_6=(1:2:1)$,
and $p_7=(1:2:2)$.
Sketch $\bbX$ in the affine plane $D_+(X_0)$ as follows:
$$
{\scriptstyle
\begin{array}{ccccc}
(0,2)& \bullet &         & \bullet & (2,2)\\
     & \bullet & \bullet & \bullet &  \\
(0,0)& \bullet & \bullet & &  (2,0)
\end{array}}
$$
The Hilbert function of $\bbX$ is
$\HF_\bbX:\ 1\ 3\ 6\ 7\ 7\cdots$
and $r_\bbX=3$. We also have $\Delta_\bbX =1$ and
the scheme $\bbX$ is a Cayley-Bacharach scheme.
A calculation yields
$$
\begin{array}{ll}
  \HF_{\delta_\bbX} & :\ 0\ 0\ 0\ 0\ 3\ 6\ 7\ 7\cdots \\
  \HF_{\fC_\bbX(-r_\bbX)} & :\ 1\ 4\ 6\ 7 \ 7\cdots.
\end{array}
$$
Since $\HF_{\delta_\bbX}(r_\bbX)=\HF_{\delta_\bbX}(3)= 0$,
the scheme $\bbX$ is not arithmetically Gorenstein.
Furthermore, we have
$$
({\rm tr}(\fC_\bbX))_1
=(\delta_\bbX)_{r_\bbX+1}\cdot(\fC_\bbX)_{-r_\bbX}
=(\delta_\bbX)_{4}\cdot(\fC_\bbX)_{-3}
=\mathfrak{m}_1.
$$
Hence $\bbX$ is a nearly Gorenstein scheme.
An application of Proposition~\ref{propSec5.5}
implies that $\bbX$ is an almost Gorenstein scheme.
In this case we do not have $\fC_\bbX =
\langle (\fC_\bbX)_{-r_\bbX}\rangle_R$,
since
$$
\dim_K (\fC_\bbX(-3))_1 =4>3
=\dim_K (\mathfrak{m}_1)
= \dim_K ((\fC_\bbX(-3))_0\mathfrak{m}_1).
$$
Moreover, if we let $p'_7 =(1:2:0)$ and
$\bbY=\{p_1,\dots,p_6, p'_7\}\subseteq \bbP^2_\bbQ$,
then the set $\bbY$ satisfies $\HF_\bbY =\HF_\bbX$, but
it is not an almost Gorenstein scheme,
since it is not a Cayley-Bacharach scheme.
\end{exam}

When $\bbX$ is a Cayley-Bacharach scheme,
the following proposition provides a necessary
and sufficient condition for $\bbX$ to be almost
Gorenstein.

\begin{prop}\label{propSec5.8}
Let $K$ be an infinite field, and let
$\bbX$ be a 0-dimensional locally Gorenstein
scheme in~$\bbP^n_K$.
Suppose that $\bbX$ is a Cayley-Bacharach scheme.
\begin{enumerate}
  \item We have $\HF_{\delta_\bbX^\sigma}(i)\le\HF_\bbX(i-r_\bbX)$
  for all $i\in\bbZ$.
  In particular, the scheme $\bbX$ is arithmetically Gorenstein
  if and only if $\HF_{\delta_\bbX^\sigma}(i)=\HF_\bbX(i-r_\bbX)$
  for all $i\in\bbZ$.

  \item $\bbX$ is an almost Gorenstein scheme
  if and only if $\HF_{\delta_\bbX^\sigma}(r_\bbX+1)=\HF_\bbX(1)$.
\end{enumerate}
\end{prop}

\begin{proof}
Since $\bbX$ is a Cayley-Bacharach scheme,
we have $\HF_{\delta_\bbX^\sigma}(i) = \HF_\bbX(i-r_\bbX)$
for $i < r_{\bbX}$ or $i \ge 2r_{\bbX}$
by Proposition~\ref{propSec3.8}.
Hence it suffices to consider the case
$r_{\bbX} \le i < 2r_{\bbX}$.
Note that $K$ is infinite. By \cite[Remark~4.13]{KL},
there exists a homogeneous element
$g\in(\fC_\bbX^\sigma)_{-r_\bbX}$ such that
${\rm Ann}_R(g)=\langle 0\rangle$.
Then we have $g\cdot (\delta_\bbX^\sigma)_i
\subseteq R_{i-r_{\bbX}}$. This implies
$\HF_{\delta_\bbX^\sigma}(i) \le \HF_\bbX(i-r_\bbX)$
for $r_{\bbX} \le i < 2r_{\bbX}$.
Moreover, the additional claim of (a)
follows from Remark~\ref{remSec3.9}.

To prove (b), according to Proposition~\ref{propSec5.06}
and (a) we only need to prove that $\bbX$ is almost Gorenstein
if $\HF_{\delta_\bbX^\sigma}(r_\bbX)=0$ and
$\HF_{\delta_\bbX^\sigma}(r_\bbX+1)=\HF_\bbX(1)$.
In this case we have
$(\delta_\bbX^\sigma)_{r_\bbX+1}\cdot g = \mathfrak{m}_1$,
where $g\in (\fC_\bbX^\sigma)_{-r_\bbX}$ is given as above.
A similar argument as in the proof of
Proposition~\ref{propSec5.5} implies that
$\bbX$ is an almost Gorenstein scheme.
\end{proof}

Recall that a 0-dimensional scheme $\bbX\subseteq\bbP^n_K$
is called {\bf level} if the socle of the Artinian local ring
$\overline{R}=R/\langle x_0\rangle$
equals $\overline{R}_{r_\bbX}$.
According to \cite[Satz~11.6]{Kr2}, the scheme $\bbX$
is level if and only if the canonical module $\omega_R$
is generated by homogeneous elements of degree $-r_\bbX+1$.
It is also known that $\bbX$ is a Cayley-Bacharach scheme
if it is level (see \cite[Proposition~6.1]{GK}).
Furthermore, Example~\ref{examS5.07} also shows that
an almost Gorenstein scheme may not be a level scheme.

\begin{prop}
Let $K$ be an infinite field, and let
$\bbX$ be a 0-dimensional locally Gorenstein
scheme in~$\bbP^n_K$.
\begin{enumerate}
  \item If $r_\bbX =1$ then $\bbX$
  is an almost Gorenstein level scheme.

  \item If $r_\bbX =1$ and $\deg(\bbX)>2$
  then $\bbX$ has minimal Dedekind different.

  \item If $\bbX$ is level and
  $\min\{\Delta_\bbX,r_\bbX\}\ge 2$, then
  $\bbX$ is not an almost Gorenstein scheme.
\end{enumerate}
\end{prop}

\begin{proof}
(a)\ Suppose that $r_\bbX=1$ and $\bbX$
is not arithmetically Gorenstein. It is clear that
$\bbX$ is a Cayley-Bacharach scheme.
Since $K$ is infinite, \cite[Remark~4.13]{KL}
yields an element $g\in (\fC_\bbX^\sigma(-1))_{0}$
such that ${\rm Ann}_R(g)=\langle 0\rangle$.
We see that $\dim_K (\mathfrak{m}_1\cdot g)
= \deg(\bbX)= \dim_K(\fC_\bbX^\sigma(-1))_{1}$.
This implies $\mathfrak{m}_1\cdot g
= (\fC_\bbX^\sigma(-1))_{1}$.
Hence $\bbX$ is a level scheme.
Furthermore, we have
$\delta_\bbX^\sigma = \bigoplus_{i\ge 2} R_i$
by \cite[Proposition~5.8]{KL}, and thus
$\HF_{\delta_\bbX^\sigma}(r_\bbX+1)=\HF_\bbX(1)$.
Consequently, Proposition~\ref{propSec5.8} shows
that $\bbX$ is an almost Gorenstein scheme.

(b)\ Since $\deg(\bbX)>2$, we have
$\Delta_\bbX = \deg(\bbX)-1 \ge 2$.
So, $\bbX$ is not an arithmetically Gorenstein scheme.
As above, the Dedekind different satisfies
$\delta_\bbX^\sigma = \bigoplus_{i\ge 2} R_i$.
Hence $\bbX$ has minimal Dedekind different.

(c)\ Let us write
\[
\begin{array}{ll}
\HF_\bbX &: 1\ h_1\ h_2\ \cdots\
h_{r_\bbX-1}\ \deg(\bbX) \ \deg(\bbX) \cdots \\
\HF_{\fC_\bbX^\sigma(-r_\bbX)} &:
c_0\ c_1\ c_2\ \cdots\ c_{r_\bbX-1}\ \deg(\bbX)
\ \deg(\bbX) \cdots
\end{array}
\]
where $c_i = \deg(\bbX)- h_{r_\bbX-i-1}$ for
$i=1,\dots,r_\bbX-1$ and $c_0=\deg(\bbX)-1$.
Suppose that $\bbX$ is an almost Gorenstein level
scheme with $\min\{\Delta_\bbX,r_\bbX\}\ge 2$.
We choose an exact sequence
$$
0 \longrightarrow R
\stackrel{\theta}{\longrightarrow}
\fC_\bbX^\sigma(-r_\bbX) \longrightarrow
C \longrightarrow 0
$$
of graded $R$-modules so that
$\mathfrak{m}\cdot C = \langle 0\rangle$.
Set $g =\theta(1)$. For $i\ge 1$, we have
$$
\mathfrak{m}_1\cdot (\fC_\bbX^\sigma(-r_\bbX))_{i-1}
= \mathfrak{m}_i\cdot g.
$$
Since $\bbX$ is level, we have
$\fC_\bbX^\sigma =\langle (\fC_\bbX^\sigma)_{-r_\bbX}\rangle_R$
by~\cite[Satz~11.6]{Kr2}.
This implies
$$
(\fC_\bbX^\sigma(-r_\bbX))_i =
\mathfrak{m}_1\cdot (\fC_\bbX^\sigma(-r_\bbX))_{i-1}
= \mathfrak{m}_i\cdot g.
$$
Therefore the Hilbert function of $\fC_\bbX^\sigma(-r_\bbX)$
has the form
\[
\HF_{\fC_\bbX^\sigma(-r_\bbX)}:\
c_0\ h_1\ h_2\ \cdots\ h_{r_\bbX-1}\ \deg(\bbX)
\ \deg(\bbX) \cdots.
\]
It follows that
$\deg(\bbX)- 1 = c_{r_\bbX-1} =
h_{r_\bbX-1} = \deg(\bbX)- \Delta_\bbX$.
Because $\Delta_\bbX >1$, we have
$\deg(\bbX)- 1 \ne \deg(\bbX)- \Delta_\bbX$,
a contradiction.
\end{proof}

Our next corollary is an immediate consequence
of this proposition. This result also follows
from \cite[Lemma~10.2 and Theorem~10.4]{GTT}.

\begin{cor}
Let $K$ be an infinite field, and let $\bbX$ be
a 0-dimensional locally Gorenstein
scheme in~$\bbP^n_K$ such that $\Delta_\bbX\ge 2$.
Then $\bbX$ is an almost Gorenstein level scheme
if and only if $r_\bbX=1$.
\end{cor}

Finally, we are interested in the question:
if $\bbX$ is an almost Gorenstein scheme
with $r_\bbX\ge 2$, then does $\Delta_\bbX=1$ hold?
When $\bbX$ is a set of $s$ distinct $K$-rational
points in uniform position, \cite[Theorem~4.7]{Hig}
provides an affirmative answer to this question
with the help of the Biinjective Map Lemma
(cf.~\cite{Kr1}).
Recall that a set of $s$ distinct $K$-rational points
$\bbX$ is called {\bf ${\bf (i,j)}$-uniform}
if every subscheme $\bbY\subseteq \bbX$ of degree
$\deg(\bbX)-i$ satisfies $\HF_\bbY(j)=\HF_\bbX(j)$.
Notice that $\bbX$ is a Cayley-Bacharach scheme
if and only if it is $(1,r_\bbX-1)$-uniform,
and if $\bbX$ is $(i,j)$-uniform then it is also
$(i-1,j)$-uniform and $(i,j-1)$-uniform.
For further information about the uniformity
of~$\bbX$ see~\cite{GK,Kr1}.
The following proposition shows that the above
question also has an affirmative answer when
$\bbX$ is $(2,r_\bbX-1)$-uniform.

\begin{prop}
Let $\bbX =\{p_1,\dots,p_s\}\subseteq \bbP^n_K$
be a $(2,r_\bbX-1)$-uniform set of $s$ distinct
$K$-rational points. Suppose that $\bbX$ is
an almost Gorenstein scheme and $r_\bbX\ge 2$.
Then we have $\Delta_\bbX=1$.
\end{prop}

\begin{proof}
Suppose for a contradiction that $\Delta_\bbX>1$.
Since $\bbX$ is an almost Gorenstein scheme,
we choose an exact sequence
$$
0 \longrightarrow R
\stackrel{\theta}{\longrightarrow}
\fC_\bbX(-r_\bbX) \longrightarrow
C \longrightarrow 0
$$
of graded $R$-modules so that
$\mathfrak{m}\cdot C = \langle 0\rangle$,
and set $g =\theta(1)$.
We write $g = x_0^{-2r_\bbX} \widetilde{g}$
with $\widetilde{g}\in R_{r_\bbX}$.
Then $\widetilde{g}(p_j) \ne 0$ for all
$j=1,\dots,s$. For each $f\in R_i$ with $i\ge0$,
we define the value
$\eta(f):=\#\{\, j\,\mid\,1\le j\le s, f(p_j)=0\,\}$.
Clearly, we have $\eta(\ell \widetilde{g})
= \eta(\ell)$ for all $\ell \in R_1$.
Now we let $\ell_0 \in R_1$ be a non-zero
element such that $\eta(\ell_0) =
\max\{\, \eta(\ell) \,\mid\,
\ell \in R_1\setminus\{0\} \,\}$.
Since $r_\bbX\ge 2$ and $\bbX$ is a
Cayley-Bacharach scheme, there exist
at least two points $p_{j_1},p_{j_2}\in \bbX$
such that $\ell_0(p_{j_1})\ne 0$ and
$\ell_0(p_{j_2})\ne 0$.
Let $f_j\in R_{r_\bbX}$ be the separator
of $\bbX\setminus\{p_j\}$ in $\bbX$ with
$f_j(p_j)=1$ and $f_j(p_k)=0$ for $k\ne j$.
Since $\bbX$ is $(2,r_\bbX-1)$-uniform,
\cite[Proposition~3.4]{Kr1} yields that
$\{\overline{f}_{j_1}, \overline{f}_{j_2}\}$
is linearly independent in $\overline{R}_{r_\bbX}$.
Let $\Sigma = \{j_1,\dots,j_{\Delta_\bbX}\}$
be a subset of $\{1,\dots,s\}$ such that
$\overline{f}_{j_1}, \overline{f}_{j_2},
\cdots,\overline{f}_{j_{\Delta_\bbX}}$
form a $K$-basis of $\overline{R}_{r_\bbX}$.
By \cite[Corollary~1.10]{Kr3}, there exist
elements $g_{j_1},g_{j_2}\in (\fC_\bbX)_{-r_\bbX}$
of the form
$g_{j_l} = x_0^{-2r_\bbX}(f_{j_l}+
\sum_{k\notin\Sigma}\beta_{kj_l}f_k)$
for $l=1,2$, where $\beta_{kj_l} \in K$.
Letting $\widetilde{g}_{j_1} = f_{j_1}+
\sum_{k\notin\Sigma}\beta_{kj_1}f_k$,
we have $\ell_0\widetilde{g}_{j_1}\ne 0$
and $\eta(\ell_0\widetilde{g}_{j_1})
\ge \eta(\ell_0)+1$.
Thus we get
$\eta(\ell_0\widetilde{g}_{j_1})>\eta(\ell\widetilde{g})$
for all $\ell\in R_1\setminus\{0\}$.
Since $x_0$ is a non-zerodivisor of $R$,
this implies that
$0\ne \ell_0g_{j_1} \notin \mathfrak{m}_1\cdot g$.
In particular, we have
$\mathfrak{m}\cdot C \ne \langle 0\rangle$,
a contradiction.
\end{proof}

% -----------------------------------------------------------
\medskip\bigbreak
\subsection*{Acknowledgments.}
This paper is partially based on
the third author's dissertation~\cite{Lon15}.
The authors thank Ernst Kunz for his encouragement
to elaborate some results presented here.
The second author would also like to acknowledge
her financial support from OeAD.
%Last, but not least, we are extremely thankful to the referees
%for their very detailed and enlightening comments.

% -----------------------------------------------------------
\medskip\bigbreak

\end{document}